\documentclass[12pt,a4paper]{article}
\usepackage{latexsym}
\usepackage{amssymb}
\usepackage{amsmath}
\usepackage{amsthm}
\usepackage{thmtools}
\usepackage{enumitem}
\usepackage{cite}
\usepackage[all]{xy}
\usepackage{framed}
\usepackage{graphicx}
\usepackage{parskip}
\usepackage{hyperref}
\usepackage[a4paper,margin=2cm]{geometry}
\usepackage{bm}
\usepackage[utf8]{inputenc}
\usepackage{mathdots}
\usepackage{pifont}
\usepackage{marginnote}
\usepackage{xcolor}
\usepackage{marginnote}
\usepackage[normalem]{ulem}
\newcounter{tmp}

\begingroup
    \makeatletter
    \@for\theoremstyle:=definition,remark,plain\do{%
        \expandafter\g@addto@macro\csname th@\theoremstyle\endcsname{%
            \addtolength\thm@preskip\parskip
            }%
        }
\endgroup

\makeatletter
\xydef@\txt@ii#1{\vbox{\vspace*{-5pt}%
 \let\\=\cr
 \tabskip=\z@skip \halign{\relax\hfil\txtline@@{##}\hfil\cr\leavevmode#1\crcr}}}
\makeatother

\theoremstyle{definition}
\newtheorem{thm}{Theorem}
\numberwithin{thm}{section}

\newtheorem{lem}[thm]{Lemma}
\newtheorem{cor}[thm]{Corollary}
\newtheorem{defn}[thm]{Definition}
\newtheorem{propn}[thm]{Proposition}
\newtheorem*{thm*}{Theorem}
\newtheorem{notn}[thm]{Notation}

\newtheorem*{qn*}{Question}

\newtheorem{conjecture}[thm]{Conjecture}
\newtheorem{props}[thm]{Properties}

\theoremstyle{remark}
\newtheorem{rk}[thm]{Remark}
\newtheorem{rks}[thm]{Remarks}
\newtheorem{ex}[thm]{Example}

\usepackage{chngcntr}
\newcounter{parentnumber}



\DeclareMathOperator{\gr}{gr}

\newcommand{\id}{\mathrm{id}}

\newcommand{\Spec}{\mathrm{Spec}}

\renewcommand{\O}{\mathcal{O}}

\newcommand\blfootnote[1]{%
  \begingroup
  \renewcommand\thefootnote{}\footnote{#1}%
  \addtocounter{footnote}{-1}%
  \endgroup
}

{%
\begin{framed}\begin{nts}%
}%
{%
\end{nts}\end{framed}%
}

\begin{document}

\numberwithin{equation}{subsection}
\binoppenalty=\maxdimen
\relpenalty=\maxdimen

\title{Skew power series rings over a prime base ring}

\author{Adam Jones\thanks{Alan Turing Building, University of Manchester, M13 9PL, \texttt{adam.jones-5@manchester.ac.uk}} \and William Woods\thanks{Pathways Department, University of Essex, CO4 3SQ,  \texttt{billy.woods@essex.ac.uk}; corresponding author}}

\maketitle
\begin{abstract}
\noindent In this paper, we investigate the structure of skew power series rings of the form $S = R[[x;\sigma,\delta]]$, where $R$ is a complete, positively filtered ring and $(\sigma,\delta)$ is a skew derivation respecting the filtration. Our main focus is on the case in which $\sigma\delta = \delta\sigma$, and we aim to use techniques in non-commutative valuation theory to address the long-standing open question: if $P$ is an invariant prime ideal of $R$, is $PS$ a prime ideal of $S$? When $R$ has characteristic $p$, our results reduce this to a finite-index problem. We also give preliminary results in the ``Iwasawa algebra" case $\delta = \sigma - \mathrm{id}_R$ in arbitrary characteristic. A key step in our argument will be to show that for a large class of Noetherian algebras, the nilradical is ``almost" $(\sigma,\delta)$-invariant in a certain sense.\blfootnote{\emph{2010 Mathematics Subject Classification}: 16S36, 16S99, 16W20, 16W25, 16W70.}
\end{abstract}

\tableofcontents

\pagebreak

\section*{Introduction}

\subsection*{Background}

Let $R$ be a Noetherian ring, complete with respect to a filtration topology, and let $(\sigma,\delta)$ be a continuous skew derivation of $R$. Assuming that the filtration is positive and $(\sigma,\delta)$ is compatible with the filtration in the appropriate sense defined below (Definition \ref{defn: compatible}), we are able to define the \emph{skew power series ring}:

\begin{equation}
\displaystyle R[[x;\sigma,\delta]]:=\left\{\sum_{n\geq 0}{r_nx^n}:r_n\in R\right\},
\end{equation}

a complete ring whose multiplication is given by $xr=\sigma(r)x+\delta(r)$ for each $r\in R$.

We consider the question: under what conditions is $R[[x;\sigma,\delta]]$ a prime ring? More generally, we aim to address and partially prove the following conjecture: 

\begin{conjecture}\label{main}

If $P$ is a prime, $(\sigma,\delta)$-invariant ideal of $R$, and $S:=R[[x;\sigma,\delta]]$, then $SP$ is a prime ideal of $S$.

\end{conjecture}

The case when $\delta = 0$ is already known in many cases: see \cite[Proposition 2.8(ii)]{schneider-venjakob-localisations} and \cite[1.4.5]{MR}.

Note that since $SP$ is a two-sided ideal of $S$, in the notation of the conjecture, and the quotient of $S$ by $SP$ is a skew power series ring defined over $R/P$ (cf. \cite[Lemma 3.14]{letzter-noeth-skew}), this conjecture is equivalent to the statement that any skew power series ring over a prime ring is prime.

A related question was asked by Letzter \cite[3.19]{letzter-noeth-skew}. In the setting of \cite{letzter-noeth-skew}, $R$ is $\mathfrak{i}$-adically complete with respect to some ideal $\mathfrak{i}$, and we have $\sigma(\mathfrak{i}^n) \subseteq \mathfrak{i}^n$ and $\delta(\mathfrak{i}^n) \subseteq \mathfrak{i}^{n+1}$ for all $n$. Letzter asks, under the assumption that $\sigma$ extends to compatible automorphisms of $R[x;\sigma,\delta]$ and $R[[x;\sigma,\delta]]$: is $R[[x;\sigma,\delta]]$ a $\sigma$-prime ring?

We will primarily restrict ourselves to the case where $\delta\sigma = \sigma\delta$. This is satisfied in the cases of many natural skew power series rings of interest, including Iwasawa algebras: see Examples \ref{ex: iterated local SPS ring}--\ref{ex: iwasawa algebras} below. In this setting, $\sigma$ does indeed extend to compatible automorphisms of $R[x;\sigma,\delta]$ and $R[[x;\sigma,\delta]]$ by setting $\sigma(x) = x$, so Letzter's hypotheses are satisfied.

(Some of our intermediate results, which may be of independent interest, extend easily to the setting where $(\sigma,\delta)$ is a \emph{$q$-skew derivation}: that is, $\delta\sigma = q\sigma\delta$ for some unit $q\in R^\times$ satisfying $\sigma(q) = q$ and $\delta(q) = 0$. We will give the more general results where applicable.)

A positive answer to Conjecture \ref{main} would be interesting for two reasons. Firstly, corresponding results in the case of skew \emph{polynomial} rings are foundational in all developments of the theory: the similar results proved as e.g. \cite[Theorem 4.2]{irving1}, \cite[Theorem 2.2]{irving2}, \cite[Proposition 3.3]{goodearl-skew-poly-and-quantized}, \cite[Proposition 3.3]{goodearl-letzter} are all used crucially in the contexts of those papers. Secondly, the corresponding result in the case of \emph{polycyclic group algebras} is a key lemma that plays an important role in our understanding of their prime ideals: see \cite[Lemma 9]{roseblade}. Unfortunately, their non-discrete analogues -- in the world of filtered skew power series rings, and the world of completed group rings of compact $p$-adic Lie groups (i.e. \emph{Iwasawa algebras}), respectively -- are missing.

We note that Bergen and Grzeszczuk have studied similar questions for \emph{discrete} skew power series rings \cite{bergen-grzeszczuk-skew}. However, in order to ensure that their skew power series rings exist, they impose the extra hypothesis that $\delta$ should be \emph{locally nilpotent}. In general, this will not be the case for the examples of interest to us. For instance, the $\mathbb{F}_p$-linear derivation $\delta$ on $R = \mathbb{F}_p[[x]]$ defined by $\delta(x) = x^{p+1}$ is \emph{topologically} nilpotent (so that the skew power series ring $R[[y; \delta]]$ may be defined) but is not \emph{locally} nilpotent. In the setting of Iwasawa algebras, $\delta$ is almost never locally nilpotent.

\subsection*{Results}

Throughout the paper, fix a prime number $p$. Our main result, stated below, gives strong evidence for Conjecture \ref{main}.

In the statement of this result, and elsewhere in the paper, we will write $R^b[[x; \sigma, \delta]]$ for the \emph{bounded} skew power series ring over $R$, as defined in \S \ref{subsec: bounded SPSRs} below. This ring is the necessary analogue of $R[[x;\sigma,\delta]]$ in the case when the filtration is not necessarily positive, and does not necessarily contain arbitrary power series with coefficients in $R$, but only those whose coefficients satisfy an appropriate convergence condition. However, note that if $R$ is positively filtered then $R^b[[x; \sigma, \delta]] = R[[x; \sigma, \delta]]$ always, which is usually the case of greatest interest to us.

\begingroup
\setcounter{tmp}{\value{thm}}
\setcounter{thm}{0} 
\renewcommand\thethm{\Alph{thm}}
\begin{thm}\label{A}
Let $R$ be a prime Noetherian algebra over $\mathbb{Z}_p$, and $w: R\to\mathbb{Z}\cup\{\infty\}$ a complete, separated Zariskian filtration, such that $\gr_w(R)$ is commutative (and Noetherian), $w(p^n)=nw(p)\geq n$ for all $n\in\mathbb{N}$, and the non-zero graded part $\gr_w(R)_{\neq 0}$ contains a non-nilpotent element. Suppose also that $(\sigma,\delta)$ is a skew derivation on $R$ compatible with $w$ and satisfying $\sigma\delta = \delta\sigma$.
\begin{enumerate}[label=(\alph*)]
\item If $p=0$ in $R$ (i.e. $R$ is an $\mathbb{F}_p$-algebra), then there exists $N\in \mathbb{N}$ such that $R^{b}[[x^{p^N}; \sigma^{p^N}, \delta^{p^N}]]$ is prime.
\item If $\delta = \sigma - \mathrm{id}$, then there exists $N\in \mathbb{N}$ such that $R^{b}[[x_{(N)}; \sigma_{(N)}, \delta_{(N)}]]$ is prime, where $x_{(N)} = (x+1)^{p^N} - 1$, $\sigma_{(N)} = \sigma^{p^N}$ and $\delta_{(N)} = \sigma_{(N)} - \mathrm{id}$.
\end{enumerate}
\end{thm}

\textit{Remarks.} We comment on the various hypotheses of this theorem.

\begin{enumerate}[label=(\roman*),noitemsep]
\item A complete filtration $w$ on $R$ is Zariskian if and only if $\gr_w(R)$ is Noetherian \cite[Chapter II, Theorem 2.1.2 and Proposition 2.2.1]{LVO}. We have included this redundant hypothesis here and elsewhere in the paper for clarity. The hypothesis that $\gr_w(R)_{\neq 0}$ contain a non-nilpotent element is required in \S \ref{subsec: constructing w'} to choose the minimal prime there called $\mathfrak{q}$. All of these hypotheses are mild in our primary cases of interest, as in Examples \ref{ex: iterated local SPS ring}--\ref{ex: iwasawa algebras}.
\item Statement (a) makes sense because if $R$ has characteristic $p$ and $\sigma\delta=\delta\sigma$ then $(\sigma^{p^m},\delta^{p^m})$ is a skew derivation, compatible with $w$, for all $m\in\mathbb{N}$. The characteristic $0$ case appears to require different treatment, so in statement (b) (which holds in characteristic $p$ and characteristic $0$) we restrict to the case $\delta = \sigma - \mathrm{id}$, as it is the case of interest for Iwasawa algebras (see Example \ref{ex: iwasawa algebras}). Note that, if $\mathrm{char}(R)=p$ and $\delta=\sigma-\mathrm{id}$, then statements (a) and (b) are equivalent.
\end{enumerate}

Of course, if we could take $N=0$ then Conjecture \ref{main} would follow immediately, at least in positive characteristic. In characteristic 0, unfortunately, some issues arise in the case where $\delta\neq\sigma-\mathrm{id}$ that may prove difficult to resolve.

In the case where $\delta=\sigma-\mathrm{id}$, however, we can deduce an immediate corollary to Theorem \ref{A} by realising $R^{b}[[x;\sigma,\delta]]$ as a crossed product of the appropriate sub-skew power series ring with $\mathbb{Z}/p^n\mathbb{Z}$. Explicitly:
\begin{center}
$\displaystyle R^b[[x;\sigma,\sigma-\id]] = \bigoplus_{i=0}^{p^N - 1}\left( R^b[[x_{(N)}; \sigma_{(N)}, \delta_{(N)}]]\right)(x+1)^i$.
\end{center}

(See Example \ref{ex: iwasawa algebras} below for an example of this.) Now we may apply \cite[Theorem 4.4, Proposition 16.4]{passmanICP} to get:

\endgroup

\begin{cor}\label{cor: iwasawa algebra is semiprime}

Let $R$ be a prime, Noetherian algebra over $\mathbb{Z}_p$ satisfying the hypotheses of Theorem \ref{A}, and suppose that $\delta=\sigma-\mathrm{id}$. Then:

\begin{itemize}

\item If $\mathrm{char}(R)=p$, then $R^{b}[[x;\sigma,\delta]]$ has prime nilradical.

\item If $\mathrm{char}(R) = 0$, then $R^b[[x;\sigma,\delta]]$ is semiprime.\qed
\end{itemize}

\end{cor}

In fact, in the case where $\mathrm{char}(R) = 0$ in the above corollary, we should in fact be able to prove that $R^b[[x;\sigma,\sigma-\mathrm{id}]]$ is in fact prime, but we hope to include this proof in a later article which develops on the ideas in this paper further.

\begingroup
\setcounter{tmp}{\value{thm}}
\setcounter{thm}{1} 
\renewcommand\thethm{\Alph{thm}}

The proof of Theorem \ref{A} uses non-commutative valuation theory. Roughly speaking, we show that the skew power series ring is prime by showing that it carries a complete filtration $f$, with a prime associated graded ring, from which the result follows. The filtration $f = f_u$ depends on a filtration $u$ on $R$, and the appropriate choice of $u$ is given by our main technical result:

\begin{thm}\label{B}
Let $R$ be a prime Noetherian algebra over $\mathbb{Z}_p$, and $w: R\to\mathbb{Z}\cup\{\infty\}$ a complete, separated Zariskian filtration, such that $\gr_w(R)$ is commutative (and Noetherian), $w(p^n)=nw(p)\geq n$ for all $n\in\mathbb{N}$, and the non-zero graded part $\gr_w(R)_{\neq 0}$ contains a non-nilpotent element. Suppose also that $(\sigma,\delta)$ is a skew derivation on $R$ compatible with $w$ and satisfying $\sigma\delta = \delta\sigma$.

Then $(\sigma,\delta)$ extends to the Goldie ring of quotients $Q(R)$, and there exists a separated filtration $u$ on $Q(R)$ such that the embedding $(R,w)\to (Q(R),u)$ is continuous, $\gr_u(Q(R))$ is prime and Noetherian, and:
\begin{enumerate}[label=(\alph*)]
\item If $\mathrm{char}(R) = p$, then there exists $N\in \mathbb{N}$ such that $(\sigma^{p^N}, \delta^{p^N})$ is compatible with $u$.

\item If  $\delta = \sigma - \mathrm{id}$, then there exists $N\in \mathbb{N}$ such that $(\sigma^{p^N}, \sigma^{p^N} - \mathrm{id})$ is compatible with $u$.
\end{enumerate}

\end{thm}

\endgroup

The key difficulty in the proof of Theorem \ref{B} is the construction of the filtration $u$. Roughly speaking, we follow the construction of \cite[Theorem C]{ardakovInv}, which uses techniques in non-commutative algebra to construct well behaved filtrations on various ring extensions of $Q(R)$. Our major addition to the argument is the involvement of the skew derivation.

We will occasionally need to consider skew derivations $(\sigma,\delta)$ on rings $R$ that are artinian but not semisimple. In these cases, we hope to be able to pass to $R/N(R)$, where $N(R)$ is the (prime) radical of $R$, in order to use the well-established theory of semisimple artinian rings. However, while $N(R)$ is preserved by $\sigma$, it will not necessarily be preserved by $\delta$ even in the nicest cases, as the following counterexample (in which $\sigma$ is the identity automorphism) shows:

\begin{ex}
\cite[Introduction]{bergen-grzeszczuk-radicals} Let $R = \mathbb{F}_p[X]/(X^p)$, so that $N(R) = (X)$. Then there exists an $\mathbb{F}_p$-linear derivation $\delta$ such that $\delta(X) = 1 \not\in N(R)$.
\end{ex}

Our final technical result, however, shows that the situation is not hopeless:

\begingroup
\setcounter{tmp}{\value{thm}}
\setcounter{thm}{2} 
\renewcommand\thethm{\Alph{thm}}

\begin{thm}\label{C}
Let $R$ be a Noetherian algebra over a field, $(\sigma,\delta)$ a skew derivation on $R$ with $\sigma\delta = \delta\sigma$, and $I$ a minimal $\sigma$-prime ideal of $R$.
\begin{enumerate}[label=(\alph*)]
\item If $R$ is artinian, and $\mathrm{char}(R) = 0$, then $\delta(I) \subseteq I$.
\item If $\mathrm{char}(R) = p > 0$, then there exist $J \supseteq I$ and $M\in \mathbb{N}$ such that
\begin{enumerate}[label=(\roman*)]
\item $J$ is a minimal $\sigma^{p^M}$-prime ideal of $R$,
\item $I = \bigcap_{n\in \mathbb{Z}} \sigma^n(J)$, and this intersection is finite,
\item $\delta^{p^M}(J) \subseteq J$.\qed
\end{enumerate}
\end{enumerate}
\end{thm}

This is enough to imply that $N(R)$ is a $(\sigma,\delta)$-ideal in case (a), and $N(R)$ is a $(\sigma^{p^M},\delta^{p^M})$-ideal for some $M$ in case (b).

\endgroup

\subsection*{Applications}

There are many examples of classes of rings whose prime quotients satisfy the hypotheses of our Theorem \ref{A}, and which are independently of great interest. We hope that this will ensure the broad applicability of our results. The examples mentioned below are discussed in detail in \S \ref{subsec: applications}:

\begin{enumerate}
\item many iterated local skew power series rings,
\item completed group algebras (Iwasawa algebras) of uniform pro-$p$ groups.
\end{enumerate}

\subsection*{Structure of the paper}

After introducing some preliminary definitions and results in \S\ref{section: prelim}, we will show in \S\ref{section: thm A} that, in the case where $Q(R)$ carries a suitable filtration $u$ satisfying certain compatibility relations with $(\sigma,\delta)$, and such that $\gr_u(Q(R))$ is itself a prime ring, we can deduce that the skew power series ring $R^b[[x;\sigma,\delta]]$ is prime. This demonstrates how we will deduce Theorem \ref{A} from Theorem \ref{B}. In \S\ref{section: thm C}, we will take an interlude and describe invariant ideals under skew derivations, in both characteristic 0 and positive characteristic, which require very different treatment. We will conclude \S 3 by proving Theorem \ref{C}. Finally, in \S\ref{section: thm B} we will prove Theorem \ref{B}, and we will conclude \S\ref{section: thm B}.4 with a proof of Theorem \ref{A}. Detailed calculations of our example applications follow in \S \ref{subsec: applications}.

In a subsequent work, we hope to build on these ideas further, with a view to proving that $R^b[[x;\sigma,\delta]]$ is prime in more general cases, such as the cases in which $\sigma$ and $\delta$ commute or $q$-commute. We do hope that the techniques employed in this paper can be developed, and we have observed that with a small refinement of them it should be possible to prove that $R^b[[x;\sigma,\sigma-\mathrm{id}]]$ is prime in characteristic $0$, but in general we believe that new methods are needed.

\noindent\textbf{Acknowledgements:} The authors are grateful to Konstantin Ardakov for his very thorough and helpful comments on an early draft of this paper. They would also like to thank the Heilbronn Institute for Mathematical Research for supporting and partially funding this work.

\section{Preliminaries}\label{section: prelim}

\subsection{Skew derivations and ideals}

Let $R$ be a Noetherian ring. A \emph{skew derivation} on $R$ is a pair $(\sigma,\delta)$, where $\sigma$ is an automorphism of $R$ and $\delta$ is a left $\sigma$-derivation of $R$, by which we mean that $\delta(ab) = \delta(a)b + \sigma(a)\delta(b)$ for all $a, b\in R$.

It is not currently possible to deal with all skew derivations in full generality, so we focus on the important special case where $\sigma$ and $\delta$ commute. The following formula can then be obtained by induction on $n$.

\begin{lem}\label{lem: binomial delta^n}
Suppose that $\sigma\delta = \delta\sigma$. Given $a,b\in R$, $n\in\mathbb{N}$, we have
\begin{equation*}
\delta^n(ab) = \sum_{k=0}^n \binom{n}{k} \delta^k \sigma^{n-k}(a) \delta^{n-k}(b).
\end{equation*}\qed
\end{lem}

(As will be standard throughout this paper, for ease of notation, we will often write function composition as concatenation: $\delta^k \sigma^{n-k} := \delta^k \circ \sigma^{n-k}$.)

\begin{rk}\label{rk: can raise skew derivations to pth powers in char p}
$ $

\begin{enumerate}[label=(\roman*)]
\item If $R$ has characteristic $p$, and $(\sigma,\delta)$ is a skew derivation on $R$ satisfying $\sigma\delta = \delta\sigma$, then for any $m$, we have that $(\sigma^{p^m}, \delta^{p^m})$ is a skew derivation on $R$. This follows immediately from Lemma \ref{lem: binomial delta^n}, noting that the binomial coefficients $\binom{p}{i} \equiv 0$ (mod $p$) for $1\leq i\leq p-1$.
\item Regardless of the characteristic of $R$, if we are in the special case $\delta = \sigma - \mathrm{id}$, then for any $n$, we have that $(\sigma^n, \sigma^n - \mathrm{id})$ is a skew derivation on $R$.
\end{enumerate}
\end{rk}

We remind the reader of the following standard definitions.

\begin{defn}
$ $

\begin{enumerate}[label=(\roman*)]
\item An ideal $I$ of $R$ is called a \emph{$\sigma$-ideal} if $\sigma(I)\subseteq I$. We may also say that such an ideal is \emph{preserved by} $\sigma$. We define similarly a \emph{$\delta$-ideal} (i.e. $\delta(I) \subseteq I$) and a \emph{$(\sigma,\delta)$-ideal} (i.e. both $\sigma(I)\subseteq I$ and $\delta(I)\subseteq I$).
\item A $\sigma$-ideal $I$ of $R$ is \emph{$\sigma$-prime} if, whenever any two $\sigma$-ideals $A$ and $B$ satisfy $AB\subseteq I$, we must have either $A\subseteq I$ or $B\subseteq I$. (We similarly define \emph{$\delta$-prime} and \emph{$(\sigma,\delta)$-prime}.)
\end{enumerate}
\end{defn}

\begin{rks}\label{rks: sigma-prime alternative characterisations}
$ $

\begin{enumerate}[label=(\roman*)]
\item As $R$ is Noetherian, an ideal $I$ is $\sigma$-prime if and only if $I = \bigcap_{n\in\mathbb{Z}} \sigma^n(P)$ for some prime ideal $P$, and in fact this $\sigma$-orbit is finite \cite[Remarks 4* and 5*]{goldie-michler}.
\item The claim that a proper $\sigma$-ideal $I$ is $\sigma$-prime is equivalent to the following condition: let $a,b\in R$, and suppose that $\sigma^n(a)Rb\subseteq I$ for all $n\in\mathbb{Z}$; then we must have either $a\in I$ or $b\in I$. This is proved in \cite[Lemma 2.1(a)]{goodearl-skew-poly-and-quantized}.
\end{enumerate}
\end{rks}

\subsection{Filtrations, associated graded rings, and skew derivations}

If $R$ is an arbitrary (even Noetherian) ring admitting a skew derivation $(\sigma,\delta)$, it is \emph{not} generally the case that we can define the skew power series ring $R[[x; \sigma,\delta]]$ without some extra hypotheses, due to convergence issues. We will usually be interested in the case when $R$ is a \emph{complete filtered} ring. This is similar to the ``$\mathfrak{i}$-adic" case studied in \cite{letzter-noeth-skew}, but in some sense very different from the ``locally nilpotent" case studied in \cite{bergen-grzeszczuk-radicals}.

The following definitions are all well known, but will set up our notational conventions for the paper. Our standard reference for filtered and graded rings is \cite{LVO}.

\begin{defn}
A \emph{(ring) filtration} on $R$ is a function $u: R\to \Gamma \cup \{\infty\}$, where $\Gamma$ is an ordered group, satisfying the following properties:
\begin{enumerate}[label=(\alph*),noitemsep]
\item $u(0) = \infty$,
\item $u(x+y) \geq \min\{u(x), u(y)\}$ for all $x,y\in R$,
\item $u(xy) \geq u(x) + u(y)$ for all $x,y\in R$.
\end{enumerate}

In the usual way, we can associate to $u$ a collection $\{F_\lambda R\}_{\lambda\in \Gamma}$ of additive subgroups of $R$ by defining $F_\lambda R = u^{-1}([\lambda,\infty])$. It will sometimes be useful to use the term \emph{filtration} to refer to this collection $\{F_\lambda R\}_{\lambda\in \Gamma}$, rather than the function $u$, and we will call $F_\lambda R$ the $\lambda$th \emph{level set} of the filtration $u$. Conversely, given the level sets $\{F_\lambda R\}$ of a filtration $u$, we can recover $u$ by setting $u(x) = \min\{\lambda : x\in F_\lambda R\}$ for $x\neq 0$.
\end{defn}

We also adopt the following conventions throughout this paper:

\begin{enumerate}[label=(\alph*),noitemsep]
\item Filtrations are always \emph{separated}: $u(x) = \infty \Leftrightarrow x = 0$.
\item $\Gamma$ will always be a discrete ordered group of rank 1, and as such we will often identify it with $\mathbb{Z}$ (or $e\mathbb{Z}$ for some $e > 0$ where necessary).
\item Our filtrations are written as \emph{descending} filtrations, i.e. if $m,n\in \Gamma$ with $m\leq n$, then $F_m R\supseteq F_n R$. (Notice that \cite{LVO} uses the opposite convention.)
\end{enumerate}

\begin{defn}
Suppose $u: R\to \Gamma \cup \{\infty\}$ is a filtration on $R$. Write $F_\lambda R = u^{-1}([\lambda,\infty])$ and $F_{\lambda^+} R = u^{-1}((\lambda,\infty])$. Then the \emph{associated graded ring} (denoted $\gr_u(R)$, or sometimes just $\gr(R)$ where the filtration $u$ is clear) is defined as
$$\gr_u(R) = \bigoplus_{\lambda\in\Gamma} (F_\lambda R / F_{\lambda^+} R).$$
\end{defn}

\begin{notn}
For any $x\in R\setminus\{0\}$, we will write $\gr_u(x)$ (or just $\gr(x)$) to denote the element $x + F_{u(x)^+} R \in \gr_u(R)$. This is sometimes called the \emph{principal symbol} of $x$ in $\gr_u(R)$.

If $(R,u)$ is a filtered ring, and $d$ is a filtered endomorphism of $R$, we define the \emph{degree} of $d$ to be the element $\deg_u(d) := \inf \{u(d(x)) - u(x) : x\in R\}\in\Gamma$ (or $\deg_u(d) = -\infty$ if this infimum does not exist, but this will not happen in this paper).
\end{notn}

The following definition will underpin most of the work of this paper.

\begin{defn}\label{defn: compatible}
Suppose $u: R\to \Gamma \cup \{\infty\}$ is a filtration on $R$, and $(\sigma,\delta)$ is a skew derivation on $R$. We say that $(\sigma,\delta)$ is \emph{compatible} with $u$ if $\deg_u(\sigma-\mathrm{id}) > 0$ and $\deg_u(\delta) > 0$. (Note that this implies that $\sigma$ and $\delta$ are filtered maps, and hence continuous with respect to the filtration topology on $R$.)

In other words: identify $\Gamma \cong e\mathbb{Z}$ for some $e > 0$. Then $(\sigma,\delta)$ is compatible with $u$ if, for all $\lambda \in \Gamma$ and all $x\in R$ with $u(x) = \lambda$, we have
\begin{itemize}[noitemsep]
\item $u(\sigma(x) - x) \geq \lambda+e$ and
\item $u(\delta(x)) \geq \lambda+e$.
\end{itemize}
\end{defn}

The importance of this definition will become clear in the next subsection. For now, we note simply that this implies $\gr_u(x) = \gr_u(\sigma(x))$ for all $x\in R$.

\textit{Remark.} This (relatively strong) notion of compatibility is useful for our purposes, but we note here that other (weaker) forms of compatibility are very common in the literature. For instance, it is common to insist simply that $\deg_u(\sigma) = 0$, rather than $\deg_u(\sigma - \mathrm{id}) > 0$.

\begin{lem}\label{lem: raising to pth powers in the char 0 case}
Let $(R,u)$ be a filtered $\mathbb{Z}_p$-algebra and $(\sigma,\sigma - \mathrm{id})$ a skew derivation on $R$ compatible with $u$. Suppose $u(p) \geq 1$ and $\deg_u(\sigma - \mathrm{id}) \geq 1$. Then $\deg_u (\sigma^{p^n} - \mathrm{id}) \geq n$. 
\end{lem}

\begin{proof}
Rewriting $\sigma$ as $(\sigma - \mathrm{id}) + \mathrm{id}$ and expanding using the binomial theorem:
\begin{align*}
\sigma^{p^n} - \mathrm{id} &= ((\sigma- \mathrm{id}) + \mathrm{id})^{p^n} - \mathrm{id}\\
&= \sum_{i=1}^{p^n} \binom{p^n}{i} (\sigma-\mathrm{id})^i.
\end{align*}
By assumption, $\deg_u(\sigma - \mathrm{id}) \geq 1$, so $\deg_u((\sigma - \mathrm{id})^i) \geq i$. Hence it remains to show that $u\left(\binom{p^n}{i}\right) \geq n - i$. Writing $v_p$ for the $p$-adic valuation on $\mathbb{Z}_p$, we have $v_p(p^n - k) = v_p(k)$ for all $1\leq k\leq p^n - 1$; and, as $$i! \binom{p^n}{i} = p^n \prod_{k=1}^{i-1} (p^n - k),$$ we may take $v_p$ of both sides to see that $v_p(i!) + v_p\left(\binom{p^n}{i}\right) = n + v_p((i-1)!)$. From here we may conclude that $u\left(\binom{p^n}{i}\right) \geq v_p\left(\binom{p^n}{i}\right) = n - v_p(i) \geq n - i$ as required.
\end{proof}

We will also need the following fact:

\begin{lem}\label{lem: compatible skew derivations can be induced up to completion}
Let $A$ be a ring admitting a filtration $u$, and let $(\sigma,\delta)$ be a skew derivation on $A$ compatible with $u$. Let $\widehat{A}$ be the completion of $A$ with respect to $u$, and $\widehat{u}$ the induced filtration. Then there is a unique extension of $(\sigma,\delta)$ to a skew derivation $(\widehat{\sigma}, \widehat{\delta})$ on $\widehat{A}$ which is compatible with $\widehat{u}$.
\end{lem}

\begin{proof}
Identify $\Gamma \cong \mathbb{Z}$. The existence of maps (of filtered abelian groups) $\widehat{\sigma}$ and $\widehat{\delta}$, and the claims that $\deg_{\widehat{u}}(\widehat{\sigma} - \mathrm{id}) \geq 1$ and $\deg_{\widehat{u}}(\widehat{\delta}) \geq 1$, follow from \cite[Chapter I, Theorem 3.4.5]{LVO}. $\widehat{\sigma}$ is an automorphism of $\widehat{A}$ by \cite[Chapter I, Corollary 3.4.8]{LVO}, and $\widehat{\delta}$ is a $\widehat{\sigma}$-derivation by an application of \cite[Chapter I, Theorem 3.4.7]{LVO}.
\end{proof}

\subsection{Microlocalisation}

Let $R$ be a ring carrying a separated filtration $w:R\to\mathbb{Z}\cup\{\infty\}$, with associated level sets $\{F_nR\}_{n\in\mathbb{Z}}$. Define the \emph{Rees ring} $\tilde{R}$ to be the graded ring given by

\begin{equation*}
\tilde{R}:=\bigoplus_{n\in\mathbb{Z}} F_nR\cdot t^{-n},
\end{equation*}

where $(at^{-n})\cdot (bt^{-m}):=abt^{-(n+m)}$. Note that $t=1\cdot t^{1}\in\tilde{R}$ is central, and despite the notation, $t$ is not actually a unit in $\tilde{R}$.

\noindent Recall from \cite[Definition 2.1.1]{LVO} that $w$ is a \emph{Zariskian filtration} if the Rees ring $\tilde{R}$ is Noetherian, and the Jacobson radical of the ring $F_0R$ contains the ideal $F_1R$. In the case where $R$ is complete with respect to $w$, it follows from \cite[Theorem 2.1.2]{LVO} that $w$ is Zariskian if and only if the associated graded ring $\gr_w R$ is Noetherian, so throughout this paper our filtrations are usually Zariskian.

So, suppose $w$ is a Zariskian filtration. Following the approach in \cite{li-ore-sets}, we will show in this section how we can lift a homogeneous localisation of the associated graded ring $\gr_w R$ to an Ore localisation of $R$, in a process known as \emph{microlocalisation}.

Firstly, we will explore how to realise both the ring $R$ and its associated graded ring as quotients of the Rees ring:

\begin{lem}\label{kernel}
There exists a surjective ring homomorphism $\rho_1:\tilde{R}\to R$ defined by $rt^{-n}\mapsto r$ with kernel $(t-1)\tilde{R}$, and a surjective homomorphism of graded rings $\rho_2:\tilde{R}\to$ $\gr_w(R)$ defined by $rt^{-n}\mapsto r+F_{n+1}R$ with kernel $t\tilde{R}$.
\end{lem}

\begin{proof}
It is clear that $\rho_1$ and $\rho_2$ are well defined, surjective ring homomorphisms, so it remains to calculate their kernels.

Clearly $(t-1)\tilde{R}$ is contained in $\ker(\rho_1)$, since $\rho_1(t)=1=\rho_1(1)$. On the other hand, if $\displaystyle \rho_1\left(\sum_{n=-N}^{M}{r_nt^{-n}}\right)=0$ for a given set of elements $r_n\in F_nR$ with $r_{-N},r_M\neq 0$, then $\displaystyle \sum_{n=-N}^M{r_n}=0$, and so since $\displaystyle r_{-N}=-\sum_{n=-N+1}^M{r_n}\in F_{-N+1}R$, we see that $r_{-N}t^{N}\equiv r_{-N}t^{N-1}$ (mod $(t-1)\tilde{R})$.

Therefore, $\displaystyle \sum_{n=-N}^{M}{r_nt^{-n}}\equiv (r_{-N}+r_{-N+1})t^{N-1}+\sum_{n=-N+2}^{M}{r_nt^{-n}}$ (mod $(t-1)\tilde{R}$), so by induction on $-N$ we may assume that $\displaystyle \sum_{n=-N}^{M}{r_nt^{-n}}\equiv st^{-M}$ (mod $(t-1)\tilde{R}$) for some $s\in F_MR$. Thus $s=\rho_1(s t^{-M})=0$, and $\displaystyle \sum_{n=-N}^{M}{r_nt^{-n}}\in (t-1)\tilde{R}$ as required.

Since $\rho_2$ is graded, it remains to check that $\rho_2(rt^{-n})=0$ for $r\in F_nR$ if and only if $rt^{-n}\in t\tilde{R}$. But $\rho_2(rt^{-n})=r+F_{n+1}R$ is zero if and only if $r\in F_{n+1}R$, i.e. if and only if $rt^{-n}=(rt^{-(n+1)})t\in t\tilde{R}$ as required.
\end{proof}

Now, let $T$ be a multiplicatively closed, homogeneous, left Ore-localisable subset of $\gr_w(R)$, let $T^{-1} \gr_w(R)$ be the corresponding localisation, and let $\alpha: \gr_w(R)\to T^{-1} \gr_w(R)$ be the natural map, so that every element of $T^{-1} \gr_w(R)$ can be written as $\alpha(X)^{-1}\alpha(A)$ for some $A\in \gr_w(R)$, $X\in T$. 

Let $S$ be the saturated lift $S:=\{r\in R:\gr(r)\in T\}$ of $T$ to $R$, and let $\tilde{S}$ be the saturated lift $\tilde{S}:=\{s\in\tilde{S}:s \text{ homogeneous, } \rho_2(s)\in T\}$ of $T$ to $\tilde{R}$. It follows immediately that $\rho_1(\tilde{S})=S$.

\begin{lem}
$S$ and $\tilde{S}$ are left Ore-localisable subsets of $R$ and $\tilde{R}$ respectively.
\end{lem}

\begin{proof}
This follows from \cite[Lemma 2.1]{li-ore-sets}.
\end{proof}

Therefore, we can define the localisations $\tilde{S}^{-1}\tilde{R}$ and $S^{-1}R$, and let $\tau:R\to S^{-1}R$ and $\tilde{\tau}:\tilde{R}\to\tilde{S}^{-1}\tilde{R}$ be the natural maps, so that every element of $S^{-1}R$ has the form $\tau(s)^{-1}\tau(r)$ for some $r\in R$, $s\in S$, and similarly for $\tilde{S}^{-1}\tilde{R}$. Also, $\tilde{S}^{-1}\tilde{R}$ is a graded ring, where 

\begin{center}
$(\tilde{S}^{-1}\tilde{R})_n:=\{\tilde{\tau}(s)^{-1}\tilde{\tau}(r):r\in\tilde{R},s\in\tilde{S}, r$ homogeneous, $\deg(r)-\deg(s)=n\}$.
\end{center}

Here, $\deg$ means the degree within the graded ring $\tilde{R}$, e.g. $\deg(t)=-1$.

From now on, \textbf{we will assume that $\tau$ is an injection}, i.e. that $sr\neq 0$ for all $s\in S$, $r\in R\setminus \{0\}$. This is always true in the case where $R$ is a prime ring.

In this case, $S^{-1}R$ is a ring extension of $R$, so we may identify $R$ with its image under $\tau$ and simply write $\tau(r)$ as $r$. Moreover, since $\tilde{S}$ is homogeneous, it follows that if $\tilde{s}\tilde{r}=0$ for some $\tilde{s}\in\tilde{S}$, $\tilde{r}\in\tilde{R}$, then $\tilde{s}=st^{-n}$ for some $s\in S$. But it follows from the proof of \cite[Lemma 3.3]{ardakovInv} that $S$ consists of regular elements: therefore, $sr\neq 0$ whenever $r\neq 0$, and so we see that $\tilde{r}=0$. Hence $\tilde{\tau}$ is also an injection, and we may consider $\tilde{S}^{-1}\tilde{R}$ a ring extension of $\tilde{R}$.

However, we cannot assume that the natural graded map $\alpha: \gr_w(R)\to T^{-1}\gr_w(R)$ is an injection: in fact in general it will not be, even when $R$ is prime.

\begin{propn}\label{propn: Rees projection maps extend to the localised Rees ring}
The maps $\rho_1,\rho_2$ extend to ring homomorphisms on $\tilde{S}^{-1}\tilde{R}$, namely $\rho_1:\tilde{S}^{-1}\tilde{R}\to S^{-1}R$ and $\rho_2:\tilde{S}^{-1}\tilde{R}\to T^{-1}\gr_w(R)$, with kernels $(t-1)\tilde{S}^{-1}\tilde{R}$ and $t\tilde{S}^{-1}\tilde{R}$ respectively, and $\rho_2$ is graded.
\end{propn}

\begin{proof}
Define $\rho_1:\tilde{S}^{-1}\tilde{R}\to S^{-1}R$ by $\rho_1(\tilde{s}^{-1}\tilde{r})=\rho_1(\tilde{s})^{-1}\rho_1(\tilde{r})$. We show first that this is well defined. Note that $\rho_1(\tilde{S})=S$, so it makes sense to take the inverse of $\rho_1(\tilde{s})$; and if $\tilde{s}^{-1}\tilde{r}=\tilde{s}'^{-1}\tilde{r}'$, then we may find $\tilde{u}\in\tilde{R}$, $\tilde{v}\in\tilde{S}$ such that $\tilde{u}\tilde{r}=\tilde{v}\tilde{r}'$ and $\tilde{v}^{-1}\tilde{u} = \tilde{s}'\tilde{s}^{-1}$. This gives $\rho_1(\tilde{u})\rho_1(\tilde{r})=\rho_1(\tilde{v})\rho_1(\tilde{r}')$, which we may rearrange to see that $\rho_1(\tilde{s})^{-1}\rho_1(\tilde{r})=\rho_1(\tilde{s}')^{-1}\rho_1(\tilde{r}')$. A similar proof shows that $\rho_2:\tilde{S}^{-1}\tilde{R}\to T^{-1} \gr_w(R)$, $\tilde{s}^{-1}\tilde{r}\mapsto\alpha(\rho_2(\tilde{s}))^{-1}\alpha(\rho_2(\tilde{r}))$ is well defined.

It is readily checked that $\rho_1,\rho_2$ are ring homomorphisms, and surjectivity is obvious.

Again, it is clear that $(t-1)\tilde{S}^{-1}\tilde{R}$ is contained in the kernel of $\rho_1$, and if $\rho_1(\tilde{s})^{-1}\rho_1(\tilde{r})=0$ then $\rho_1(\tilde{r})=0$, so $\tilde{r}\in (t-1)\tilde{S}^{-1}\tilde{R}$ by Lemma \ref{kernel}, and hence $\tilde{s}^{-1}\tilde{r}\in (t-1)\tilde{S}^{-1}\tilde{R}$. The same argument shows that the kernel of $\rho_2$ is $t\tilde{S}^{-1}\tilde{R}$, so it remains to prove that $\rho_2$ is graded. 

Given $\tilde{s}^{-1}\tilde{r}$ with $\tilde{r}$ homogeneous and $\deg(\tilde{r})-\deg(\tilde{s})=n$, we know that $\rho_2(\tilde{r})=0$ or $\deg(\rho_2(\tilde{r}))-\deg(\rho_2(\tilde{s}))=n$ in $\gr_w(R)$. But the natural localisation map $\alpha: \gr_w(R)\to T^{-1} \gr_w(R)$ is graded, so if $X\in \gr_w(R)$ is homogeneous of degree $d$ then either $X\in\ker(\alpha)$ and $\alpha(X)=0$, or else $\alpha(X)$ has degree $d$ in $T^{-1} \gr_w(R)$. 

In particular, if $\rho_2(\tilde{r})$ lies in the kernel of $\alpha$ then $\rho_2(\tilde{s}^{-1}\tilde{r})=0$, otherwise $\deg(\alpha(\rho_2(\tilde{r})))=\deg(\rho_2(\tilde{r}))$. Furthermore, since $\rho_2(\tilde{s})\in T$, it does not lie in the kernel of $\alpha$, so  $\deg(\alpha(\rho_2(\tilde{s})))=\deg(\rho_2(\tilde{s}))$, and hence we can deduce that $\deg(\alpha(\rho_2(\tilde{r})))-\deg(\alpha(\rho_2(\tilde{s})))=n$ in $T^{-1} \gr_w(R)$, hence this extension of $\rho_2$ is graded.\end{proof}

Using this proposition, we define a new filtration $w'$ on $S^{-1}R$ whose $n$th level set is given by $F_n'S^{-1}R:=\rho_1((\tilde{S}^{-1}\tilde{R})_n)$.

\begin{propn}\label{propn: Rees and associated graded of microlocalisation}
The map $w':S^{-1}R\to\mathbb{Z}\cup\{\infty\}$ is a filtration with Rees ring $\tilde{S}^{-1}\tilde{R}$ and associated graded ring $\gr_{w'}(S^{-1}R) = T^{-1}\gr_w(R)$. Moreover, $w'$ satisfies:
\begin{itemize}
\item For all $x\in S^{-1}R$, $w'(x)=\min\{w(r)-w(s):r\in R,s\in S,x=s^{-1}r\}$.
\item For all $r\in R$, $w'(r)\geq w(r)$, with equality if $r\in S$.
\item For all $r\in R$, $s\in S$, $w'(s^{-1}r)=w'(r)-w(s)$.
\end{itemize}
\end{propn}

\begin{proof}
To prove that $w'$ is a filtration, we need only prove that the level sets $F_n'S^{-1}R$ satisfy $F_{n+1}'S^{-1}R\subseteq F_n'S^{-1}R$ and $F_n'S^{-1}R\cdot F_m'S^{-1}R\subseteq F_{n+m}'S^{-1}R$ for all $n,m\in\mathbb{Z}$.

The second property is clear since $F_n'S^{-1}R=\rho_1((\tilde{S}^{-1}\tilde{R})_n)$, we have that $(\tilde{S}^{-1}\tilde{R})_n\cdot(\tilde{S}^{-1}\tilde{R})_m\subseteq (\tilde{S}^{-1}\tilde{R})_{n+m}$, and $\rho_1$ is a ring homomorphism. For the first, note that if $x\in F_{n+1}'S^{-1}R$ then $x=\rho_1(\tilde{s}^{-1}\tilde{r})$  for some $\tilde{r}\in\tilde{R}$ homogeneous, $\tilde{s}\in\tilde{S}$ with $\deg(\tilde{r})-\deg(\tilde{s})=n+1$, and since $t-1$ lies in the kernel of $\rho_1$, it follows that $x=\rho_1(\tilde{s}^{-1}t\tilde{r})$, and $\deg(t\tilde{r})-\deg(\tilde{s})=\deg(\tilde{r})-\deg(\tilde{s})-1=n$. Thus $x\in F_n'S^{-1}R$ and hence $F_{n+1}'S^{-1}R\subseteq F_n'S^{-1}R$.

Let $\displaystyle \widetilde{S^{-1}R}=\bigoplus_{n\in\mathbb{Z}} (F_n'S^{-1}R)\overline{t}^{-n}$ be the Rees ring, and define a map $\widetilde{S^{-1}R}\to\tilde{S}^{-1}\tilde{R}$ by $$(s^{-1}r)\overline{t}^{-n}\mapsto (st^{-a})^{-1}(rt^{-b})$$ when $w'(s)=a$, $w'(r)=b$ and $b-a=n$. This is a well defined ring isomorphism.

Similarly, we can define a map $\gr_{w'}(S^{-1}R)\to T^{-1}\gr_w(R)$ by
$$s^{-1}r +F_{n+1}'S^{-1}R\mapsto \alpha(s+F_{a+1}R)^{-1}\alpha(r+F_{b+1}R),$$ which is an isomorphism of graded rings.

Now, given $x\in S^{-1}R$, we see that
\begin{align*}
w'(x)&=\min\{n\in\mathbb{Z}:x\in\rho_1((\tilde{S}^{-1}\tilde{R})_n)\}\\
&=\min\{n\in\mathbb{Z}:x=\rho_1((st^{-a})^{-1}(rt^{-b})), s\in S, w(s)\geq a, w(r)\geq b, b-a=n\}\\
&=\min\{n\in\mathbb{Z}:x=\rho_1((st^{-a})^{-1}(rt^{-b})),s\in S,w(s)=a,w(r)=b,b-a\leq n\}\\
&\qquad\qquad\text{(because if $w(r)>b$ then $\rho_1(rt^{-b})=\rho_1(rt^{-(b+1)})$)}\\
\phantom{w'(x)}&=\min\{n\in\mathbb{Z}:x=s^{-1}r,w(r)-w(s)\leq n\}\\
&=\min\{w(r)-w(s):r\in R,s\in S,x=s^{-1}r\}
\end{align*}
as required.

Now, if $r\in R$ and $w(r)=n$ then $r=\rho_1(rt^{-n})\in\rho_1(\tilde{R}_n)$, so $r\in\rho_1((\tilde{S}^{-1}\tilde{R})_n)$ and hence $w'(r)\geq n=w(r)$. Moreover, if $w'(r)>w(r)$ then under the natural map $\alpha$, $r+F_{n+1}R$ is sent to $r+F_{n+1}'S^{-1}R=0$, i.e. $r+F_nR$ lies in the kernel of $\alpha$. So by standard properties of the localisation, there exists $a\in T$ such that $a(r+F_{n+1}R)=0$, i.e. there exists $s\in S$ such that $sr\in F_{n+w(s)+1}R$, and hence $w(rs)>w(r)+w(s)$.

In particular, if $r\in S$ then this means that $b:=r+F_{n+1}R\in T$, and it follows that $ab=0$ with $a,b\in T$, and hence $0\in T$ -- contradiction. Therefore if $r\in S$, $w'(r)=w(r)$.

Finally, if $s\in S$ with $w(s)=n$ then $s+F_{n+1}'S^{-1}R$ is a unit in $\gr_{w'} (S^{-1}R)$, and hence we see that $w'(sr)=w'(s)+w'(r)$ for all $r\in R$. It follows that $w'(s^{-1}r)=w'(r)-w'(s)$ for all $r\in R$ as required. \end{proof}

\textbf{Note:} It is proved in \cite{li-ore-sets} that $w'$ is in fact a Zariskian filtration on $S^{-1}R$.

Now, suppose that $R$ carries a skew derivation $(\sigma,\delta)$ such that $(\sigma,\delta)$ is compatible with the initial filtration $w$. We will prove that $(\sigma,\delta)$ extends to a skew derivation of $S^{-1}R$, and that the extension is compatible with $w'$.

Firstly, it is standard that any automorphism $\sigma$ and any $\sigma$-derivation $\delta$ extend uniquely to a localisation, by defining $\sigma(s^{-1}r)=\sigma(s)^{-1}\sigma(r)$, and $\delta(s^{-1}r)=\sigma(s)^{-1}(\delta(r)-\delta(s)s^{-1}r)$, so it remains to prove that the extension is compatible.

We can also extend $\sigma$ and $\delta$ to the Rees ring $\tilde{R}$ by:

\begin{center}
$\tilde{\sigma}\left(\underset{n\in\mathbb{Z}}{\sum}{r_nt^{-n}}\right)=\underset{n\in\mathbb{Z}}{\sum}{\sigma(r_n)t^{-n}}$, $\tilde{\delta}\left(\underset{n\in\mathbb{Z}}{\sum}{r_nt^{-n}}\right)=\underset{n\in\mathbb{Z}}{\sum}{\delta(r_n)t^{-n}}$
\end{center}

\begin{lem}\label{Rees-skew derivation}
The pair $(\tilde{\sigma},\tilde{\delta})$ is a graded skew derivation of $\tilde{R}$ such that $\sigma(t)=t$ and $\tilde{\delta}(t)=0$. In particular, the extensions of $\tilde{\sigma}$ and $\tilde{\delta}$ to $\tilde{S}^{-1}\tilde{R}$ are graded, and they preserve the ideals generated by $t$ and $t-1$. 
\end{lem}

\begin{proof}
This follows from the fact that $(\sigma,\delta)$ is compatible with $w$.\end{proof}

\begin{thm}\label{thm: compatible extension}
The natural extension of $(\sigma,\delta)$ to $S^{-1}R$ is compatible with $w'$.
\end{thm}

\begin{proof}
By Lemma \ref{Rees-skew derivation}, $(\tilde{\sigma}',\tilde{\delta}')$ is a skew derivation of $\tilde{S}^{-1}\tilde{R}$ which preserves the ideals generated by $t$ and $t-1$, so it follows from Proposition \ref{propn: Rees projection maps extend to the localised Rees ring} and Proposition \ref{propn: Rees and associated graded of microlocalisation} that it induces a skew derivation $(\sigma',\delta')$ of $S^{-1}R$, and this will coincide with the extension of $(\sigma,\delta)$ to $S^{-1}R$.

To prove compatibility, note that if we let $d$ be either $\delta$ or $\sigma-\mathrm{id}$, then $w(d(r))>w(r)$ for all $r\in R$ by compatibility, from which it follows that $\tilde{d}(rt^{-n})=d(r)t^{-n}=t(d(r)t^{-(n+1)})\in t\tilde{R}$. Thus $d(\tilde{R})\subseteq t\tilde{R}$, and hence $\tilde{d}'(\tilde{S}^{-1}\tilde{R})\subseteq t\tilde{S}^{-1}\tilde{R}$, and thus the induced map $d'$ is zero on the associated graded ring $\gr_{w'} S^{-1}R=\tilde{S}^{-1}\tilde{R}/(t)$. Therefore $(\sigma',\delta')$ is compatible with $w'$. \end{proof}

\subsection{Bounded skew power series rings}\label{subsec: bounded SPSRs}

Suppose we are given a skew derivation $(\sigma,\delta)$ on a ring $R$. The \emph{skew polynomial ring} $R[x;\sigma,\delta]$ is defined to be equal to $R[x]$ as a left $R$-module, with the (unique) ring structure determined by extending the rule
\begin{equation}\label{eqn: multiplication}
xa = \sigma(a)x + \delta(a)
\end{equation}
(for all $a\in R$) to an $R$-linear multiplication map $R[x]\times R[x] \to R[x]$.

We would like to form the \emph{skew power series ring} in the same way, beginning with the left $R$-module $R[[x]]$ and providing it with a unique ring structure by extending the rule (\ref{eqn: multiplication}) to a (continuous) $R$-linear multiplication. If this ring structure does indeed exist, we call the ring a \emph{skew power series ring}, and denote it $R[[x;\sigma,\delta]]$. However, in general, this ring structure may fail to exist due to convergence issues. To fix this, we will need some extra hypotheses on $R$, $\sigma$ and $\delta$: 

\begin{propn}\label{propn: skew-filtration}
Suppose that $u: R\to \mathbb{Z} \cup \{\infty\}$ is a separated filtration on $R$, that $R$ is complete with respect to $u$, and that $(\sigma,\delta)$ is a skew derivation on $R$ compatible with $u$.

\begin{enumerate}[label=(\roman*)]
\item The left $R$-module
$$R^b[[x;\sigma,\delta]] = \left\{ \sum_{n\geq 0} r_n x^n : r_n\in R,\, u(r_n) + \tfrac12 n \to \infty \text{ as } n \to \infty\right\}$$
is in fact a ring, with multiplication given by extending the rule (\ref{eqn: multiplication}) to a continuous $R$-linear multiplication in a unique way.
\item Define the function $f_u: R^b[[x;\sigma,\delta]] \to \frac12 \mathbb{Z}\cup\{\infty\}$ as follows: for all choices of $r_i\in R$, set
$$f_u\left(\sum_{i=0}^\infty r_i x^i\right) = \inf_{i\geq 0} \left\{u(r_i) + \tfrac{1}{2}i \right\}.$$
Writing $r\in R$ as the ``constant" series $r + 0x + 0x^2 + \dots\in R^b[[x;\sigma,\delta]]$, we may identify $R$ with a subring of $R^b[[x;\sigma,\delta]]$. Then $f_u$ is a ring filtration with $f_u|_R = u$, $R^b[[x;\sigma,\delta]]$ is complete with respect to $f_u$, and the identity automorphism on $\gr_u(R)$ extends to an isomorphism of graded rings $\gr_{f_u} (R^b[[x;\sigma,\delta]]) \to (\gr_u (R))[Z]$ mapping $x$ to $Z$.
\end{enumerate}
\end{propn}

\begin{proof}
Part (i) follows exactly as in \cite[\S 3.4]{letzter-noeth-skew}, replacing $\mathfrak{i}^n$ with $F_n R$, the $n$th level set of the filtration $u$. Part (ii) then follows as in \cite[Lemma 1.13, Remark 1.14]{woods-SPS-dim}, except for the claim that $R^b[[x;\sigma,\delta]]$ is complete. To show this, take a sequence of elements $s^{(j)} = \sum_{i\geq 0} r^{(j)}_i x^i\in R^b[[x;\sigma,\delta]]$ such that $f_u(s^{(j)}) \to \infty$ as $j\to\infty$. Then, as $j\to\infty$, by definition we have $\inf_{i\geq 0}\{u(r^{(j)}_i) + \frac12 i\} \to \infty$, from which we may conclude that $u(r^{(j)}_i) + \frac12 i\to \infty$ for each $i$, and hence $u(r^{(j)}_i) \to\infty$ for each $i$. But as $R$ is complete with respect to $f_u|_R = u$, we are done.
\end{proof}

We call $R^b[[x;\sigma,\delta]]$ a \emph{bounded skew power series ring}. Note that if $R$ is positively filtered, i.e. $u(r)\geq 0$ for all $r\in R$, then the requirement that $u(r_n) + \frac12 n \to \infty$ becomes vacuous. In this case, we just call the resulting ring the \emph{skew power series ring}, denoted by $R[[x;\sigma,\delta]]$, and it contains all power series in $R$.

\begin{rks}
$ $

\begin{enumerate}[label=(\roman*)]
\item If $T$ is any positively filtered, $(\sigma,\delta)$-invariant subring of $R$, then $T[[x;\sigma,\delta]]$ is a subring of $R^b[[x;\sigma,\delta]]$.
\item If $\gr_u(R)$ is a prime ring, then $\gr_{f_u}(R^b[[x;\sigma,\delta]]) \cong \gr_u(R)[Z]$ is also a prime ring.
\item It is worth mentioning that the bounded skew power series ring is often far too large for many purposes. For instance, if $G$ is a compact $p$-adic Lie group, the appropriate notion of its Iwasawa algebra over $\mathbb{Q}_p$ is usually taken to be $\mathbb{Q}_p G := (\mathbb{Z}_p G)[\tfrac{1}{p}]$, an \emph{incomplete} ring. However, if a complete filtered ring is prime, then any dense subring is also prime, so our results for bounded skew power series rings can also provide information about these incomplete subrings.
\end{enumerate}
\end{rks}

\section{Extending prime ideals}\label{section: thm A}

In this section, we will explore how to prove under certain conditions that the bounded skew power series ring $R^{b}[[x;\sigma,\delta]]$ is prime, which will be important for the proof of our main result Theorem \ref{A}.

\subsection{Ideals in skew power series rings}

Let $R$ be a ring, carrying a complete, separated filtration $w:R\to\mathbb{Z}\cup\{\infty\}$, let $(\sigma,\delta)$ be a skew derivation, compatible with $w$, and let $S=R^b[[x;\sigma,\delta]]$ be the bounded skew power series ring. Note that the well-definedness of the multiplicative structure on $S$ depends on some notion of compatibility between $(\sigma,\delta)$ and the topology of $R$, but the following result ensures that we can often change the filtration without affecting the ring structure. In the statement, we use the notation $f_u$ from Proposition \ref{propn: skew-filtration}.

\begin{propn}\label{propn: independence of filtration}

Suppose that $R$ carries a complete filtration $v:R\to\mathbb{Z}\cup\{\infty\}$ such that the identity map $(R,w)\to (R,v)$ is continuous and $(\sigma,\delta)$ is compatible with $v$. Then the identity map $(S,f_w)\to (S,f_v)$ is a continuous ring isomorphism with respect to the multiplication defined using $w$ and the multiplication defined using $v$.

\end{propn}

\begin{proof}

From the definition of the filtrations $f_w$ and $f_v$, it is clear that the identity map $S\to S$ is a continuous, additive bijection. It follows immediately that it is multiplicative, since its restriction to the skew polynomial ring $R[x;\sigma,\delta]$ is multiplicative, and $R[x;\sigma,\delta]$ is dense in $R^b[[x;\sigma,\delta]]$.\end{proof}

Now, let us suppose that $w$ is a Zariskian filtration. In this case, for any closed two-sided ideal $I$, the quotient ring $R/I$ is complete with respect to the natural quotient filtration $\overline{w}(r+I)=\sup\{w(r+y):y\in I\}$, which is still a Zariskian filtration.

\begin{lem}\label{quotient}
Let $I$ be a closed, $(\sigma,\delta)$-invariant ideal of $R$, and let $(\overline{\sigma},\overline{\delta})$ be the induced skew derivation of $R/I$, then $(\overline{\sigma},\overline{\delta})$ is compatible with the quotient filtration $\overline{w}$. The closure $\overline{IS}$ is a two-sided ideal of $S$ and $S/\overline{IS}\cong (R/I)^b[[x;\overline{\sigma},\overline{\delta}]]$.
\end{lem}

(Compare \cite[Lemma 3.14]{letzter-noeth-skew}.)

\begin{proof}

Let $d$ be either $\sigma-\mathrm{id}$ or $\delta$. To prove that $(\overline{\sigma},\overline{\delta})$ is compatible with $\overline{w}$, we just need to show that $\overline{w}(\overline{d}(r+I))>\overline{w}(r+I)$ for all $r\in R$. But
\begin{align*}
&&\overline{w}(\overline{d}(r+I)) &= \overline{w}(d(r) + I)\\
&&&= \sup_{y\in I} w(d(r) + y)\\
&&&\geq \sup_{z\in I} w(d(r)+d(z))& \text{as } d(I) \subseteq I\\
&&&> \sup_{z\in I} w(r+z)& \text{as } d \text{ has positive } w \text{-degree}\\
&&&= \overline{w}(r+I).
\end{align*}

Clearly $\overline{IS}$ is a closed, right ideal of $S$. Consider the natural map 

\begin{center}
$R^b[[x;\sigma,\delta]]\to (R/I)^b[[\overline{x};\overline{\sigma},\overline{\delta}]],rx^n\mapsto (r+I)\overline{x}^n$.
\end{center} 

This is clearly a surjective ring homomorphism, and its kernel is the set of all power series $\displaystyle \sum_{n\in\mathbb{N}} r_n x^n$ such that $r_n\in I$ for all $n$, which is clearly equal to the closure of $IS$ in $S$. Therefore $\overline{IS}$ is a two-sided ideal, and the quotient $S/\overline{IS}$ is isomorphic to $(R/I)^b[[\overline{x};\overline{\sigma},\overline{\delta}]]$ as required.\end{proof}

Hence, given a prime $(\sigma,\delta)$-ideal $P$ of $R$, this lemma tells us that the induced ideal $\overline{PS}$ is prime in $S$ if and only if the skew power series ring $(R/P)^b[[x; \sigma,\delta]]$ is a prime ring. So we may reduce to the case $P = 0$ throughout.

\textbf{Note:} If we assume that $w$ is positive, then $f_w$ is also positive and Zariskian, so it follows from \cite{LVO}[Corollary 2.1.5] that all one-sided ideals in $R[[x;\sigma,\delta]]$ are closed, so we may just write $\overline{PS}$ as $PS$.

\subsection{Criterion for primeness}

Again, let $R$ be a ring carrying a separated filtration $v:R\to e\mathbb{Z}\cup\{\infty\}$. First consider the following well known result.

\begin{lem}\label{gr-prime}
If the associated graded ring $\gr_v R$ is a prime ring, then $R$ is prime.
\end{lem}

\begin{proof}

Suppose that $I,J$ are two-sided ideals of $R$ such that $IJ=0$, then let $\gr I$ be the additive subgroup of $\gr_v R$ generated by elements of the form $\gr(y)$ for $y\in I$, and let $\gr J$ be defined similarly. Then $\gr I$ and $\gr J$ are two-sided ideals of $\gr R$, and $(\gr I)(\gr J)=0$. So since $\gr R$ is prime, we may assume that $\gr I=0$, and hence $\gr(y)=0$ for all $y\in I$.

So since $v$ is a separated filtration, this implies that $I=0$, and hence $R$ is prime.\end{proof}

Using Proposition \ref{propn: skew-filtration}, we see that the associated graded ring of $R^b[[x;\sigma,\delta]]$ with respect to $f_v$ is $(\gr_v R)[X]$, so it follows that if $\gr_v R$ is prime then $\gr_{f_v} R^b[[x;\sigma,\delta]]$ is also prime, and hence $R^b[[x;\sigma,\delta]]$ is a prime ring by the lemma.

So, from now on, we will assume that $R$ is a prime Noetherian ring, and we will assume that $w:R\to\mathbb{Z}\cup\{\infty\}$ is a Zariskian filtration, and that $(\sigma,\delta)$ is compatible with $w$. Let $Q(R)$ be the simple Goldie ring of quotients of $R$, and recall that there is a unique extension of $(\sigma,\delta)$ to a skew derivation of $Q(R)$.

\textbf{Assumption:} We suppose that $Q(R)$ carries a separated filtration $u$ such that $(\sigma,\delta)$ is compatible with $u$, the natural inclusion map $(R,w)\to (Q(R),u)$ is continuous, and $\gr_u Q(R)$ is a prime, Noetherian ring.

Let $Q$ be the completion of $Q(R)$ with respect to $u$. Then $(\sigma,\delta)$ extends uniquely to a skew derivation of $Q$ by Lemma \ref{lem: compatible skew derivations can be induced up to completion}, and this extension will still be compatible with $u$. Since $\gr_u Q= \gr_u Q(R)$ is prime and Noetherian, and $\gr_{f_u} Q^b[[x;\sigma,\delta]]\cong (\gr_u Q)[Z]$, it follows from Lemma \ref{gr-prime} that $Q^b[[x;\sigma,\delta]]$ is prime and Noetherian.

From now on, set $T:=Q^b[[x;\sigma,\delta]]$. Note that since the inclusion map $(R,w) \to  (Q(R),u)$ is continuous, it follows that the identity $(R,w)\to (R,u)$ is continuous, so using Proposition \ref{propn: independence of filtration} we see that $R^b[[x;\sigma,\delta]]$ is a subring of $T$. Using this, we will prove that the skew power series ring $S:=R^b[[x;\sigma,\delta]]$ is prime.

\begin{propn}\label{two-sided}
If $I$ is a closed, two-sided ideal of $S$ then $\overline{TI}$ is a two-sided ideal of $T$.
\end{propn}

\begin{proof}
Clearly $\overline{TI}$ is a left ideal of $T$ and $\overline{TI}S\subseteq \overline{TI}$, so it remains to prove that $\overline{TI}T\subseteq \overline{TI}$.

Since we know that $\overline{TI}x^n\subseteq \overline{TI}$ for all $n\in\mathbb{N}$, we only need to prove that $\overline{TI}q\subseteq \overline{TI}$ for all $q\in Q$. Since $Q(R)$ is dense in $Q$ and $\overline{TI}$ is closed in $T$, we can assume that $q\in Q(R)$, so $q=rs^{-1}$ for some $r,s\in R$, $s$ regular.

But $\overline{TI}r\subseteq \overline{TI}$ so we only need to prove that $\overline{TI}s^{-1}\subseteq \overline{TI}$.

Consider the chain of left ideals $\overline{TI}\supseteq \overline{TI}s\supseteq \overline{TI}s^2\supseteq\cdots$ in $T$. Then multiplying on the right by $s^{-n}$ for any $n\in\mathbb{N}$ gives $\overline{TI}\subseteq \overline{TI}s^{-1}\subseteq \overline{TI}s^{-2}\subseteq\cdots$, and since $T$ is Noetherian, this chain must terminate. So $\overline{TI}s^{-(n+1)}=\overline{TI}s^{-n}$ for some $n$, and it follows that $\overline{TI}s^{-1}=\overline{TI}$ as required.
\end{proof}

\begin{cor}\label{conclusion}
$R^b[[x;\sigma,\delta]]$ is a prime ring.
\end{cor}

\begin{proof}
Suppose $S$ has closed ideals $I,J$ such that $IJ=0$. Then by Proposition \ref{two-sided}, $\overline{TI}$ and $\overline{TJ}$ are two-sided ideals of $T$, and $(\overline{TI})(\overline{TJ})\subseteq \overline{TIJ}=0$. But we know that $T$ is a prime ring, so either $\overline{TI}=0$ or $\overline{TJ}=0$, meaning that either $I=0$ or $J=0$ as required.
\end{proof}

So, altogether, we have proved the following key result:

\begin{thm}\label{thm: key for Thm A}
Let $R$ be a prime, Noetherian ring carrying a complete Zariskian filtration $w:R\to\mathbb{Z}\cup\{\infty\}$ and a compatible skew derivation $(\sigma,\delta)$. Suppose further that the Goldie ring of quotients $Q(R)$ carries a filtration $u$ such that:

\begin{itemize}

\item $(R,w)\to (Q(R),u)$ is continuous.

\item $\gr_u Q(R)$ is prime and Noetherian.

\item The unique extension of $(\sigma,\delta)$ to $Q(R)$ is compatible with $u$.

\end{itemize}

Then the skew power series ring $R^b[[x;\sigma,\delta]]$ is prime.
\end{thm}

For the remainder of this paper, we will focus on constructing this filtration $u$.

\section{Minimal $\sigma$-prime ideals}\label{section: thm C}

Continue to suppose that $R$ is a Noetherian ring admitting a skew derivation $(\sigma,\delta)$.

\subsection{The characteristic 0 case}

We begin by reminding the reader of the following two results, both essentially taken from \cite{bergen-grzeszczuk-radicals}. In the following, if $q$ is a central invertible element of $R$ and $n\in\mathbb{N}$, we will write the useful element $\{n!\}_q := (1)(1+q)(1+q+q^2)\dots(1+q+\dots+q^{n-1})\in R$.

\begin{lem}\cite[Lemma 3(i)]{bergen-grzeszczuk-radicals}\label{lem: I + delta(I) is an ideal}
Let $R$ be a ring, and $(\sigma,\delta)$ a skew derivation on $R$. If $I$ is a $\sigma$-ideal of $R$, then $I + \delta(I)$ is an ideal of $R$.\qed
\end{lem}

Recall that the \emph{prime radical} of a ring $R$ is the intersection of all its (minimal) prime ideals \cite[4.10.13]{lam-first-course}; equivalently, as $R$ is Noetherian, it is the unique largest nilpotent two-sided ideal of $R$ \cite[4.10.30]{lam-first-course}.

\begin{propn}\label{propn: q-skew case when delta preserves nilpotent ideals}
Let $R$ be a Noetherian ring, $(\sigma,\delta)$ a $q$-skew derivation on $R$ for some central $q\in R^\times$, and $N$ any $\sigma$-ideal of $R$. If $\delta^n(N^n) = 0$ for some positive integer $n$, and $\{n!\}_q$ is \emph{invertible} in $R$, then $\delta(N)^n \subseteq N$. In particular, if $N$ is the prime radical of $R$, then $\delta(N) \subseteq N$.
\end{propn}

\begin{proof}
The following calculation is essentially identical to the one given in \cite[Lemma 4(iii)]{bergen-grzeszczuk-radicals}: if $s_1, \dots, s_n\in N$ are arbitrary and $r_i = \sigma^{i-n}(s_i)$ for all $1\leq i\leq n$, then
\begin{align*}
0 = \delta^n(r_1 r_2 \dots r_n) &\in \;\;\; \{n!\}_q \times \sigma^{n-1}\delta(r_1) \sigma^{n-2}\delta(r_2) \dots \sigma\delta(r_{n-1})\delta(r_n) + N\\
&= \dfrac{\{n!\}_q}{q^{(n-1)n/2}} \delta\sigma^{n-1}(r_1) \delta\sigma^{n-2}(r_2) \dots \delta\sigma(r_{n-1})\delta(r_n) + N\\
&= \dfrac{\{n!\}_q}{q^{(n-1)n/2}} \delta(s_1) \delta(s_2) \dots \delta(s_{n-1})\delta(s_n) + N,
\end{align*}
so if $\{n!\}_q$ is invertible, we see that $\delta(s_1) \delta(s_2) \dots \delta(s_{n-1})\delta(s_n) \in N$, i.e. that $\delta(N)^n \subseteq N$.

Now suppose that $N$ is the prime radical of $R$ (a $\sigma$-ideal of $R$, as the automorphism $\sigma$ permutes the minimal prime ideals of $R$): there exists $n$ such that $\delta^n(N^n) = 0$ as $N$ is a nilpotent ideal \cite{lam-first-course}. Now Lemma \ref{lem: I + delta(I) is an ideal} shows that $J := N + \delta(N)$ is an ideal of $R$. Since $J^m \subseteq N + \delta(N)^m$ for any $m$, we see that $J^n \subseteq N$ is a nilpotent ideal, and so $J \subseteq N$.
\end{proof}

Many examples satisfy the hypotheses of the above proposition, but we single out two broad and interesting classes in particular:

\begin{cor}\label{cor: prime radical preserved in char 0}
Suppose $R$ is a Noetherian algebra over a central subring $A$ which is a domain, and $(\sigma,\delta)$ is a $q$-skew derivation on $R$ for some $q\in A^\times$. Assume also \textit{either} that $q$ is not a root of unity \textit{or} that $\mathrm{char}(R) = 0$ and $q = 1$. Then $\delta$ preserves the prime radical $N$ of $R$.
\end{cor}

\begin{proof}
Either hypothesis on $q$ implies that $1 + q + \dots + q^m \neq 0$ for all $m\geq 0$, and so $\{n!\}_q \neq 0\in A$ for all $n$.
\end{proof}

\subsection{Calculations in characteristic $p > 0$}

In positive characteristic, the issue is less straightforward: if $R$ is an algebra over some finite field $\mathbb{F}_{p^m}$, and $q\in \mathbb{F}_{p^m}^\times$, then $\{n!\}_q$ will only be nonzero for sufficiently small $n$. We are particularly interested in the case $q = 1$.

We now fix a prime $p$ and turn to the case in which $R$ is an $\mathbb{F}_p$-algebra. In the characteristic $p$ case, it is generally not true that $\delta$ will preserve the prime radical of $R$ (see the counterexample in \cite[Introduction]{bergen-grzeszczuk-radicals}); as will be a recurring theme throughout the paper, we will have to replace $(\sigma,\delta)$ by something else.

\begin{notn}
Given $n\in\mathbb{N}$, we can write $n = a_0 + a_1p + \dots + a_r p^r$, where $0\leq a_i < p$ for all $i$. We will write $[n] = [n]_p$ to denote the element $(a_0, a_1, \dots)$ of $\mathbb{N}^\infty$, where $a_i = 0$ for all $i > r$. We say that $[n]$ and $[m]$ \emph{share no common component} if there is no $i\in\mathbb{N}$ such that the $i$th entry of both $[n]$ and $[m]$ is nonzero.
\end{notn}

\begin{lem}\label{lem: calculation of delta^n with no common components}
Assume $\sigma\delta = \delta\sigma$. Suppose we are given $i,j,k, n\in\mathbb{N}$ such that $[i] + [j] + [k] = [n]$. Then $i+j+k = n$. Furthermore, for any $a, b, x\in R$, there exist $\alpha_{i,j,k}\in \mathbb{F}_p^\times$ such that
\begin{align}\label{eqn: expanding delta^n}
\delta^n(axb) = \sum_{[i] + [j] + [k] = [n]} \alpha_{i,j,k} \delta^i\sigma^{n-i}(a) \delta^j\sigma^{k}(x) \delta^k(b).
\end{align}
\end{lem}

\begin{proof}
The first statement is clear, since if $[i]+[j]+[k]=[n]$ then the $p$-adic coefficients of $i, j$ and $k$ sum to the corresponding coefficients of $n$. (The converse to this statement is false, of course.)

For the second statement, let $r_n$ be the index of the final nonzero entry of $[n]$, or $-1$ if $n = 0$. When $r_n = -1$ there is nothing to prove, as $\delta^0(axb) = axb$. So let us take an integer $t\geq 1$ such that the result is known to hold for all $m$ with $r_m < t$, i.e. all $m = a_0 + a_1 p + \dots + a_s p^s$ with $s < t$. We will prove the result for arbitrary $n = m + p^t y$ for some $m,y\in\mathbb{N}$ with $r_m < t$ and $y < p$, i.e. for all $n$ with $r_n < t+1$, and then we will be done by induction.

By the inductive hypothesis,
\begin{align}\label{eqn: expanding delta^m}
\delta^n(axb) = \delta^{p^t y}(\delta^m(axb)) = \delta^{p^t y} \left( \sum_{[i] + [j] + [k] = [m]} \alpha_{i,j,k} \delta^i\sigma^{m-i}(a) \delta^j\sigma^{k}(x)\delta^k(b)\right).
\end{align}
Using the fact that $\delta^{p^t}$ is a $\sigma^{p^t}$-derivation, as in Remark \ref{rk: can raise skew derivations to pth powers in char p}, we may now apply Lemma \ref{lem: binomial delta^n} twice to each term on the right-hand side to get a ``trinomial" expansion. Ignoring the scalars $\alpha_{i,j,k}$ for now, we can calculate that
\begin{align*}
\delta^{p^t y} ( a'x'b' ) = \sum_{u+v+w=y} \binom{y}{u,v,w} \delta^{p^t u} \sigma^{p^t(v+w)}(a') \delta^{p^t v}\sigma^{p^t w}(x') \delta^{p^t w}(b'),
\end{align*}
into which we can substitute $a' = \delta^i \sigma^{m-i}(a)$, $x' = \delta^j \sigma^{k}(x)$ and $b' = \delta^k(b)$ for each $[i]+[j]+[k]=[m]$, to get
\begin{align}\label{eqn: expanding delta^(p^t y)}
&\delta^{p^t y} ( \delta^i\sigma^{m-i}(a) \delta^j\sigma^{k}(x)\delta^k(b))\nonumber\\& = \sum_{u+v+w=y} \binom{y}{u,v,w} \delta^{i+p^t u} \sigma^{m - i + p^t(v+w)}(a) \delta^{j+p^t v}\sigma^{k + p^t w}(x) \delta^{k+p^t w}(b).
\end{align}
Upon substituting equation (\ref{eqn: expanding delta^(p^t y)}) into equation (\ref{eqn: expanding delta^m}), we get a sum over the indexing set
$$S = \{(i,j,k,u,v,w) : [i]+[j]+[k] = [m] \text{ and } u+v+w = y\}.$$
However, what we want, as in equation (\ref{eqn: expanding delta^n}), is a sum over the indexing set 
$$T = \{(i',j',k') : [i']+[j']+[k'] = [n]\}.$$
It is easy to see that setting $i' = i + p^t u$, $j' = j + p^t v$ and $k' = k + p^t w$ gives an explicit bijection $S\to T$, and that this gives $n - i' = m - i + p^t(v+w)$. Setting $\displaystyle \beta_{i',j',k'} = \alpha_{i,j,k} \binom{y}{u,v,w}\in\mathbb{F}_p^\times$ and making all of these substitutions in equations (\ref{eqn: expanding delta^m}--\ref{eqn: expanding delta^(p^t y)}), we get
$$\delta^n(axb) = \sum_{[i'] + [j'] + [k'] = [n]} \beta_{i',j',k'} \delta^{i'}\sigma^{n-{i'}}(a) \delta^{j'}\sigma^{k'}(x)\delta^{k'}(b),$$
as required.
\end{proof}

\begin{cor}\label{cor: calculating delta^(r+s) for minimal s}
Assume $\sigma\delta = \delta\sigma$. Let $I$ be an ideal, $a, b\in I$, and $x\in R$. Suppose also that there are $r, s$ such that $\delta^r(a)\not\in I, \delta^s(b)\not\in I$, and take $r$ and $s$ to be the minimal such integers. If $[r]$ and $[s]$ share no common component, then there exists some $\alpha\in\mathbb{F}_p^\times$ such that $\delta^{r+s}(axb) \equiv \alpha \delta^r\sigma^{s}(a) \sigma^{s}(x)\delta^{s}(b) \;\; (\text{mod } I)$.
\end{cor}

\begin{proof}
Lemma \ref{lem: calculation of delta^n with no common components} implies that $\delta^{r+s}(axb)$ is a linear combination of elements of the form $\delta^i\sigma^{r+s-i}(a) \delta^j\sigma^{k}(x) \delta^k(b)$, where $i+j+k = r+s$. But if we have either $i < r$ or $k < s$, then $\delta^i\sigma^{r+s-i}(a) \delta^j\sigma^{k}(x) \delta^k(b)\in I$, so after reducing modulo $I$, the only term of interest is the term corresponding to $i = r$, $j = 0$ and $k = s$. As $[r]$ and $[s]$ share no common component, we have that $[i]+[j]+[k] = [r+s]$ for this term, and so in the notation of Lemma \ref{lem: calculation of delta^n with no common components} we have $\alpha = \alpha_{r,0,s} \neq 0$.
\end{proof}

\subsection{The $\delta$-core of an ideal in characteristic $p > 0$}

\begin{notn}
Suppose $I$ is any $\sigma$-ideal. $I$ will in general not be preserved by $\delta$, and to account for this, we would like to define $$\delta\text{-core}(I) = \{a\in I : \delta^n(a)\in I \text{ for all } n\geq 0\}.$$ It can be checked from the definition that this is the largest $(\sigma,\delta)$-ideal contained within $I$.

However, $\delta\text{-core}(I)$ will usually be too small for our purposes, so we also make the following definitions. As $R$ has characteristic $p$, we may take $(\sigma^{p^m},\delta^{p^m})$ to be our skew derivation of interest rather than $(\sigma,\delta)$. $I$ is still a $\sigma^{p^m}$-ideal, and we may then define $\delta^{p^m}\text{-core}(I)$, the largest $(\sigma^{p^m},\delta^{p^m})$-ideal contained within $I$, for all $m\geq 0$.

This is an ascending sequence, in the sense that $\delta^{p^m}\text{-core}(I) \subseteq \delta^{p^{m+1}}\text{-core}(I)$ for all $m\geq 0$: so, finally, we will denote the directed limit by $$\delta^{p^\infty}\text{-core}(I) = \bigcup_{m\in\mathbb{N}} \delta^{p^m}\text{-core}(I).$$
\end{notn}

\begin{lem}\label{lem: raise delta to pth power to preserve a sigma-ideal}
Fix a $\sigma$-ideal $I$ of a Noetherian ring $R$ of characteristic $p$. Then there exists $M\geq 0$ such that $\delta^{p^M}\text{-core}(I) = \delta^{p^\infty}\text{-core}(I)$.
\end{lem}

\begin{proof}
The sequence of ideals $\delta\text{-core}(I) \subseteq \delta^p\text{-core}(I) \subseteq \delta^{p^2}\text{-core}(I)\subseteq \dots$ stabilises, as $R$ is Noetherian.
\end{proof}

We look now at the special case where $I$ is $\sigma$-prime and $M = 0$ in the above lemma.

\begin{propn}\label{propn: delta-cores are sigma-prime}
Suppose that $I$ is a $\sigma$-prime ideal with the property that $\delta\text{-core}(I) = \delta^{p^\infty}\text{-core}(I)$. Then $\delta\text{-core}(I)$ is a $\sigma$-prime ideal.
\end{propn}

\begin{proof}
Write $J = \delta\text{-core}(I)$. Suppose, for contradiction, that $J$ is not $\sigma$-prime: then there exist elements $a,b\in R\setminus J$ such that $\sigma^{n}(a)Rb \subseteq J$ for all $n\in\mathbb{Z}$. As $J\subseteq I$, and $I$ is $\sigma$-prime, we must have either $a\in I$ or $b\in I$ by Remark \ref{rks: sigma-prime alternative characterisations}(ii). We will treat the case $b\in I$: the case $a\in I$ is similar.

As $b\in I\setminus J$, there exists some minimal $s\in\mathbb{N}$ such that $\delta^{s}(b) \not\in I$: fix this $s$. Choose also arbitrary $n\in\mathbb{Z}$ and $x\in R$. By Lemma \ref{lem: binomial delta^n}, we have $\delta^{s}(\sigma^{n}(a)xb) \equiv \sigma^{n}(a) x\delta^{s}(b)$ (mod $I$). But as $\sigma^{n}(a)Rb \subseteq J$, by definition of $J$ we have $\delta^{s}(\sigma^{n}(a)Rb)\subseteq I$, implying that $\sigma^{n}(a) R\delta^{s}(b)\subseteq I$. But as $\delta^{s}(b)\not\in I$, and $n$ was arbitrary, we must have $a\in I$.

To obtain the necessary contradiction, we will show that in fact $a\in J$. To do this, it will suffice to show that there exists some $N\geq 0$ such that $a\in \delta^{p^{N}}\text{-core}(I)$, i.e. that $\delta^{lp^N}(a)\in I$ for all $l\geq 0$. To do this, it will be enough to show that $\delta^{r'}(a) \in I$ for all $r'$ such that $[r']$ and $[s]$ have no common component: then, choosing $N$ so that $p^N > s$, it will follow that $[lp^N]$ and $[s]$ will have no common component for all $l\geq 0$. This is how we proceed.

Suppose, for contradiction, that this is not true: let $r\geq 0$ be minimal such that $[r]$ and $[s]$ have no common component but $\delta^{r}(a) \not\in I$. We may now apply Corollary \ref{cor: calculating delta^(r+s) for minimal s} to see that, for arbitrary $x\in R$ and $n\in\mathbb{Z}$,

$$\delta^{r+s}(\sigma^{n}(a)xb) \equiv \alpha \delta^r\sigma^{s + n}(a) \sigma^s(x)\delta^s(b) \;\; (\text{mod } I)$$

for some nonzero $\alpha$. Hence, as $\sigma^{n}(a)xb\in J$, we have $\delta^{r}\sigma^{s + n}(a) \sigma^{s}(x)\delta^{s}(b)\in I$.

But $x\in R$ is arbitrary here, and $\sigma$ is surjective, so this implies that $\delta^{r}\sigma^{s + n}(a) R\delta^{s}(b)\subseteq I$. Finally, rewriting this as $\sigma^n(\delta^r \sigma^s(a)) R \delta^s(b) \subseteq I$ and noting once more that $n$ was arbitrary, we can deduce from Remark \ref{rks: sigma-prime alternative characterisations}(ii) (as $\delta^s(b) \not\in I$) that we must have $\delta^r\sigma^s(a)\in I$, and hence $\delta^r(a)\in I$. This is a contradiction.
\end{proof}

\textit{Proof of Theorem \ref{C}.}

\begin{enumerate}[label=(\alph*)]
\item As $I$ is a $\sigma$-prime ideal, Remark \ref{rks: sigma-prime alternative characterisations}(i) tells us that it is an intersection of prime ideals of $R$, and hence it contains the (prime) radical $N$ of $R$. From Corollary \ref{cor: prime radical preserved in char 0}, we already know that $\delta(N) \subseteq N$. As $N$ is a $(\sigma,\delta)$-ideal, we may define the induced skew derivation $(\overline{\sigma},\overline{\delta})$ on $\overline{R} := R/N$: now $\delta(I) \subseteq I$ if and only if $\overline{\delta}(\overline{I}) \subseteq \overline{I}$. So, passing to the quotient if necessary, it will suffice to assume henceforth that $N = 0$, i.e. $R$ is a \emph{semisimple} artinian ring.

By the standard theory of semisimple artinian rings \cite[Exercise 1.1.7, Theorem 1.3.5ff.]{lam-first-course}, all prime ideals of $R$ are minimal (and maximal), and there is a finite set of centrally primitive idempotents $\{e_P : P\in\Spec(R)\}$ whose sum is $1_R$, with the properties that $e_P R \cong R/P$ and $(1-e_P)R \cong P$ as $R$-bimodules, each $e_P R$ is a simple artinian ring, and $\displaystyle R = \bigoplus_{P\in\Spec(R)} e_P R$ as a sum of two-sided ideals.

As $I$ is a minimal $\sigma$-prime of $R$, there exists a $\sigma$-orbit $\mathcal{X} = \{P_1, \dots, P_t\}$ of distinct minimal primes of $R$ such that $I = P_1 \cap \dots \cap P_t$. Write $\mathcal{X}' = \{P_{t+1}, \dots, P_s\} = \Spec(R) \setminus \mathcal{X}$, and denote by $e_i$ the central idempotent associated to $P_i$: this means that $I = e_{t+1}R \oplus \dots \oplus e_s R$. But as $\mathcal{X}'$ is a union of $\sigma$-orbits, $\sigma$ permutes the simple rings $e_{t+1}R, \dots,$ $e_s R$. Now it follows from \cite[Lemma 1.2]{cauchon-robson} that $\delta(I) \subseteq I$.

\item Firstly, fix one of the finitely many minimal primes $P$ above $I$ as in Remark \ref{rks: sigma-prime alternative characterisations}(i), so that $I = \bigcap_{n\in\mathbb{Z}} \sigma^n(P)$. Setting $I_1 = I$ and $M_0 = 0$, we see that we have a $\sigma^{p^{M_0}}$-prime ideal $I_1$.

We now construct a sequence of ideals by induction. Suppose that we have defined a $\sigma^{p^{M_{j-1}}}$-prime ideal $I_j$. Then, by Lemma \ref{lem: raise delta to pth power to preserve a sigma-ideal}, we choose the minimal $M_j \geq M_{j-1}$ such that $\delta^{p^{M_j}}\text{-core}(I_j)$ is equal to $\delta^{p^\infty}\text{-core}(I_j)$; and, by definition, the ideal $I_{j+1} = \bigcap_{n\in\mathbb{Z}} \sigma^{np^{M_j}}(P)$ is a $\sigma^{p^{M_j}}$-prime ideal.

By definition, we have $I_1 \subseteq I_2 \subseteq I_3 \subseteq \dots \subseteq P$. But as the $\sigma$-orbit of $P$ is finite by Remark \ref{rks: sigma-prime alternative characterisations}(i), and each $I_j$ is defined to be the intersection of some prime ideals in the $\sigma$-orbit of $P$, this sequence must stabilise, say at $I_k = I_{k+1} = \cdots$. Let $J = I_k$ and $M = M_k$: then $J$ is a $\sigma^{p^M}$-prime ideal satisfying $\delta^{p^M}\text{-core}(J) = \delta^{p^\infty}\text{-core}(J)$.

\begin{enumerate}[label=(\roman*)]
\item As $J$ is the intersection of the $\sigma^{p^M}$-orbit of a \emph{minimal} prime ideal, it is a \emph{minimal} $\sigma^{p^M}$-prime ideal.
\item We know already that $I = \bigcap_{n\in\mathbb{Z}} \sigma^n(P)$, and this intersection is finite. But $I\subseteq J\subseteq P$ by construction, so the same is true upon replacing $P$ by $J$.
\item We may apply Proposition \ref{propn: delta-cores are sigma-prime}, to see that $\delta^{p^M}\text{-core}(J)$ is a $\sigma^{p^M}$-prime ideal; but $\delta^{p^M}\text{-core}(J)\subseteq J$, so by (i), they must be equal.\qed
\end{enumerate}
\end{enumerate}

\section{Constructing a filtration}\label{section: thm B}

Throughout this section, $R$ will denote a \emph{prime} algebra over $\mathbb{Z}_p$, and $w: R\to \mathbb{Z}\cup\{\infty\}$ a complete, separated Zariskian filtration such that $w(p) \geq 1$ and $w(p^n)=nw(p)$ for all $n\in\mathbb{N}$. We will suppose also that $\gr_w (R)$ is commutative (and Noetherian), that the non-zero part $(\gr_w(R))_{\neq 0}$ contains a non-nilpotent element, and that $(\sigma,\delta)$ is compatible with $w$.

\textbf{Note:} If $p\neq 0$ in $R$ then the condition that $w(p^n)=nw(p)\geq n$ for all $n$ implies that gr$(p)$ is non-nilpotent of positive degree, and hence $\gr_w(R)_{\neq 0}$ contains a non-nilpotent element. So this latter condition only needs to be stated in the case where char$(R)=p$.

In this section, we give a proof of Theorem \ref{B}. Our notation broadly follows that of \cite{jones-abelian-by-procyclic}.

\subsection{Localisation and completion}\label{subsec: constructing w'}

Fix a minimal prime ideal $\mathfrak{q}\lhd \gr_w(R)$ not containing $(\gr_w(R))_{\neq 0}$, which we know exists since this set contains a non-nilpotent element. Moreover, if $p\neq 0$ in $R$, we can assume that gr$(p)$ does not lie in $\mathfrak{q}$. Then $\mathfrak{q}$ is a graded ideal of $\gr_w (R)$, so let $T$ be the set of homogeneous elements in $(\gr_w(R))\setminus \mathfrak{q}$, and let
$$S = \{r\in R \;:\; \gr_w(r) \in T\}$$
be its \emph{saturated lift} in $R$ (in the sense of \cite{li-ore-sets}). Then $S^{-1}R = Q(R)$ by \cite[Lemma 3.3]{ardakovInv}.

Using Proposition \ref{propn: Rees and associated graded of microlocalisation}, we can construct a Zariskian filtration $w' = w'_{\mathfrak{q}}$ on $Q(R)$, which is known to satisfy the following conditions by the results of \cite{li-ore-sets} and \cite{ardakovInv}.

\begin{props}\label{props: the filtration w'}
$ $

\begin{enumerate}[label=(\roman*),noitemsep]
\item We have $w'(x) = \max\{w(r) - w(s): \exists s\in S, r\in R \text{ such that } x = s^{-1}r\}$ for all $x\in Q(R)$.
\item For all $s\in S$ and $r\in R$, we have $w'(s^{-1}r) = w'(r) - w(s)$.
\item By construction, $w'(r) \geq w(r)$ for all $r\in R$, with equality if $r\in S$. In particular, $w'(p)\geq 1$.
\item $\gr_{w'}(Q(R)) \cong T^{-1}\gr_w(R)$.
\item The completion $Q'$ of $Q(R)$ with respect to $w'$ is artinian, as in \cite[\S 3.4]{ardakovInv}.
\end{enumerate}
\end{props}

We will be interested in this completed ring $Q'$, with the natural filtration induced from $w'$ (which we continue to denote by $w'$).

Recall from Theorem \ref{thm: compatible extension} that the natural extension of $(\sigma,\delta)$ to $Q(R)$ is compatible with $w'$, and hence $(\sigma,\delta)$ extends to a compatible skew derivation of the completion $Q'$ by Lemma \ref{lem: compatible skew derivations can be induced up to completion}. We will continue to denote this extension by $(\sigma,\delta)$.

Now set $E = T^{-1}\gr_w(R)$ and $\mathfrak{q}' = T^{-1} \mathfrak{q}$, and let $U = \{x\in Q': w'(x)\geq 0\}$ be the positive part of $Q'$. Write $\{F_n Q'\}_{n\in\mathbb{Z}}$ for the sequence of level sets associated to the filtration $w'$ on $Q'$.

With this notation, it is known that there exists a regular normal element $z\in J(U)$ satisfying the following:
\begin{props}\label{props: the element z}
$ $

\begin{enumerate}[label=(\roman*),noitemsep]
\item \cite[Proposition 3.4]{ardakovInv} $U$ is Noetherian and $z$-adically complete,
\item \cite[\S 3.14, Proof of Theorem C(a)]{ardakovInv} $z^nU = F_{nw'(z)}Q'$,
\item \cite[\S 3.14, Proof of Theorem C(a)]{ardakovInv} $w'(z)$ is the minimal positive degree of an element in $E/\mathfrak{q}'$. Hence we can assume that $w'(z)\leq w'(p)$, so $p\in zU$.
\end{enumerate}
\end{props}

We will now pass to the $z$-adic filtration with respect to $U$ defined on $Q'$: that is, the filtration $v_{z,U}$ defined by $v_{z,U}(x) = r$ if and only if $x\in z^r U\setminus z^{r+1} U$ for all $x\in Q'\setminus \{0\}$.

\begin{propn}\label{propn: the filtration v_{z,U}}
$ $

\begin{enumerate}[label=(\roman*)]
\item If $w'(z) = 1$, then $(\sigma,\delta)$ is compatible with $v_{z,U}$.
\item If $\mathrm{char}(Q') = p$, then $(\sigma^{p^m}, \delta^{p^m})$ is compatible with $v_{z,U}$ for any $m$ such that $p^m \geq w'(z)$.
\item If $\delta = \sigma - \mathrm{id}$, then $(\sigma^{p^m}, \sigma^{p^m}-\mathrm{id})$ is compatible with $v_{z,U}$ for any $m$ such that $m \geq w'(z)$.
\end{enumerate}
\end{propn}

\begin{proof}
The sequence of level sets associated to $v_{z,U}$ is $\{z^n U\}_{n\in\mathbb{Z}}$.

\begin{enumerate}[label=(\roman*)]
\item Suppose $w'(z) = 1$, and write $d$ for either $\sigma - \mathrm{id}$ or $\delta$. Then, as $w'$ is known to be compatible with $(\sigma,\delta)$, we see that $d(F_nQ') \subseteq F_{n+1}Q'$ for all $n$. Now Properties \ref{props: the element z}(ii) and (iii) imply that $z^n U = F_n Q'$, $z^{n+1} U = F_{n+1} Q'$, from which it follows trivially that $d(z^n U) \subseteq z^{n+1} U$ as required.
\item Choose $m$ such that $p^m \geq w'(z)$, and write $d$ for either $\sigma^{p^m} - \mathrm{id} = (\sigma - \mathrm{id})^{p^m}$ or $\delta^{p^m}$. Then we have $d(F_k Q') \subseteq F_{k + p^m} Q' \subseteq F_{k + w'(z)} Q'$ for all $k$, and Property \ref{props: the element z}(ii) implies that $z^n U = F_{nw'(z)} Q'$, $z^{n+1} U = F_{(n+1)w'(z)} Q'$, from which it again follows that $d(z^n U) \subseteq z^{n+1} U$.
\item Since $\deg_{w'}(\sigma-\mathrm{id})\geq 1$ and $w'(p)\geq 1$, it follows immediately from Lemma \ref{lem: raising to pth powers in the char 0 case} that $\deg_{w'}(\sigma^{p^m}-\id)\geq m\geq w'(z)$, and hence $\deg_{v_{z,U}}(\sigma^{p^m}-\id)\geq 1$ as required.\qedhere
\end{enumerate}
\end{proof}

\subsection{Passing to a simple quotient}\label{subsec: construction of O}

As in the previous subsection, $Q'$ is an artinian algebra over a field, $w'$ a Zariskian filtration on $Q'$, $U$ the $w'$-positive part of $Q'$, $z\in J(U)$ a regular normal element, and $v_{z,U}$ the associated $z$-adic filtration. We will now further assume that some maximal ideal $M$ of $Q'$ is a $(\sigma,\delta)$-ideal, and we will set $\widehat{Q} = Q'/M$, a \emph{simple} artinian ring. Also set $V = (U+M)/M$, and $\overline{z} = z+M$.

We now list some properties of $\widehat{Q}$, given in \cite{ardakovInv}.

\begin{props}\label{props: maximal order O}
$ $

\begin{enumerate}[label=(\roman*),noitemsep]
\item \cite[Theorem 3.11]{ardakovInv} There exists a maximal order $\mathcal{O}\subseteq \widehat{Q}$ equivalent to $V$. 
\item \cite[\S 3.6, Proposition 3.7(a)]{ardakovInv} $V\subseteq \mathcal{O} \subseteq \overline{z}^{-r}V$ for some $r\geq 0$.
\item \cite[Theorem 3.6]{ardakovInv} $\mathcal{O}$ is a prime hereditary Noetherian ring, with a unique maximal two-sided ideal $J(\mathcal{O})$, and $p\in J(\mathcal{O})$.
\item We will define $J(\mathcal{O})^n$ for all $n\in\mathbb{Z}$. Indeed, if $n\geq 0$, then the definition of $J(\mathcal{O})^n$ is standard: then we set $J(\mathcal{O})^{-1} = \{x\in\widehat{Q} : J(\mathcal{O}) x \subseteq \mathcal{O}\}$ and $J(\mathcal{O})^{-n} := (J(\mathcal{O})^{-1})^n$. The argument in \cite[proof of Proposition 3.9]{ardakovInv} shows that $J$ is left and right invertible, so it follows that $J(\mathcal{O})^a J(\mathcal{O})^b = J(\mathcal{O})^{a+b}$ for all $a,b\in\mathbb{Z}$.
\end{enumerate}
\end{props}

Eventually, we aim to equip $\widehat{Q}$ with the $J(\mathcal{O})$-adic filtration $u$: that is, the filtration defined by $u(x) = r$ if and only if $x\in J(\mathcal{O})^r \setminus J(\mathcal{O})^{r+1}$ for all $x\in\widehat{Q}\setminus \{0\}$. This filtration is of particular interest because, as proved in \cite[Theorem C]{ardakovInv}, $\gr_u(\widehat{Q})$ is a prime ring.

However, in order to obtain interesting information about skew power series over $\widehat{Q}$, we need to understand how $(\sigma,\delta)$ interacts with this filtration.

\begin{lem}\label{lem: checking compatibility of a derivation on O and J(O)}
Suppose that $\tau$ is an automorphism of $\widehat{Q}$ fixing $\mathcal{O}$, and let $d$ be a $\tau$-derivation satisfying $d(\mathcal{O}) \subseteq J(\mathcal{O})$ and $d(J(\mathcal{O})) \subseteq J(\mathcal{O})^2$. Then $d(J(\mathcal{O})^n) \subseteq J(\mathcal{O})^{n+1}$ for all $n\in\mathbb{Z}$.
\end{lem}

\begin{proof}
We are given that the statement is true for $n=0$ and $n=1$. Note that $\tau$ must preserve $J(\O)$ by the uniqueness of Property \ref{props: maximal order O}.

When $n\geq 2$, we may proceed by induction. Suppose $d(J(\mathcal{O})^{n-1}) \subseteq J(\mathcal{O})^{n}$. Let $a\in J(\mathcal{O})^{n-1}$ and $b\in J(\mathcal{O})$: then $d(ab) = d(a)b + \tau(a) d(b)$, and both terms on the right-hand side are contained in $J(\mathcal{O})^{n+1}$. But as $J(\mathcal{O})^n$ is generated by elements of the form $ab$, we see that $d(J(\mathcal{O})^n) \subseteq J(\mathcal{O})^{n+1}$.

Now choose $a\in J(\mathcal{O})^{-n}$ for $n\geq 1$. Given arbitrary $b\in J(\mathcal{O})^n$, we have $d(a)b = d(ab) - \tau(a)d(b)$, which is an element of $J(\mathcal{O})$, and hence $d(a)J(\mathcal{O})^n \subseteq J(\mathcal{O})$. Right-multiplying both sides by $J(\mathcal{O})^{-1}$ and using Property \ref{props: maximal order O}(iv) now shows that $d(a)J(\mathcal{O})^{n-1} \subseteq \mathcal{O}$, and so by definition $d(a) \in J(\mathcal{O})^{-n+1}$ as required.
\end{proof}

We will show that the filtered ring $(\widehat{Q},u)$ inherits certain near-compatibility properties from $(Q',v_{z,U})$. To do this, it will be convenient to pass through several intermediate filtrations, for which we immediately set up notation:
\begin{itemize}[noitemsep]
\item the $z$-adic filtration with respect to $U$ on $Q'$, denoted $v_{z,U}$ with level sets $\{z^n U\}_{n\in\mathbb{Z}}$,
\item the $\overline{z}$-adic filtration with respect to $V$ on $\widehat{Q}$, denoted $v_{\overline{z},V}$ with level sets $\{\overline{z}^n V\}_{n\in\mathbb{Z}}$,
\item the $\overline{z}$-adic filtration with respect to $\mathcal{O}$ on $\widehat{Q}$, denoted $v_{\overline{z},\mathcal{O}}$  with level sets $\{\overline{z}^n \mathcal{O}\}_{n\in\mathbb{Z}}$,
\item the $J(\mathcal{O})$-adic filtration on $\widehat{Q}$ as defined above, denoted $u$ with level sets $\{J(\mathcal{O})^n\}_{n\in\mathbb{Z}}$.
\end{itemize}

Note that $\overline{z}^n V = (z^n U + M)/M$: that is, $v_{\overline{z},V}$ is just the quotient filtration induced by $v_{z,U}$ on $\widehat{Q}$.

\begin{propn}\label{propn: from Q' to Q-hat, case Q-hat simple}
Suppose that $(\sigma,\delta)$ is compatible with $v_{z,U}$, and write $(\widehat{\sigma}, \widehat{\delta})$ for the skew derivation induced by $(\sigma,\delta)$ on $\widehat{Q}$. Then $(\widehat{\sigma}, \widehat{\delta})$ is compatible with $v_{\overline{z},V}$.
\end{propn}

\begin{proof}
Write $d$ for either $\delta$ or $\sigma-\mathrm{id}$, and $\widehat{d}$ for $\widehat{\delta}$ or $\widehat{\sigma}-\mathrm{id}$ respectively. By definition, we get
\[
\widehat{d}(\overline{z}^n V) = \widehat{d}((z^n U + M)/M) = (d(z^n U) + M)/M \subseteq (z^{n+1} U + M)/M = \overline{z}^{n+1} V. \qedhere
\]
\end{proof}

\begin{propn}\label{propn: raising to pth powers, case Q-hat simple}
Suppose that $(\widehat{\sigma},\widehat{\delta})$ is compatible with $v_{\overline{z},V}$. Then:
\begin{enumerate}[label=(\roman*)]
\item If $\mathrm{char}(\widehat{Q}) = p$, then there exists some $\ell$ such that $(\widehat{\sigma}^{p^\ell}, \widehat{\delta}^{p^\ell})$ is compatible with $u$.
\item If $\widehat{\delta} = \widehat{\sigma} - \mathrm{id}$, then there exists some $\ell$ such that $(\widehat{\sigma}^{p^\ell}, \widehat{\sigma}^{p^\ell} - \mathrm{id})$ is compatible with $u$.
\end{enumerate}
\end{propn}

\begin{proof}
For simplicity of notation, we will identify $\Gamma \cong \mathbb{Z}$ throughout this proof without loss of generality. Also write $\deg_V$, $\deg_\mathcal{O}$ and $\deg_{J(\O)}$ for degrees with respect to the filtrations $v_{\overline{z},V}$, $v_{\overline{z},\mathcal{O}}$ and $u$ respectively. Write $\widehat{d}$ for either $\widehat{\delta}$ or $\widehat{\sigma}-\mathrm{id}$. Cases (i) and (ii) will be treated in parallel, as their methods are very similar.

\textbf{Step 1.} By assumption, $\deg_V(\widehat{d}) \geq 1$. Property \ref{props: maximal order O}(ii) tells us that $V \subseteq \mathcal{O} \subseteq \overline{z}^{-r} V$ for some $r$, and hence, multiplying through by $\overline{z}^n$, it follows that $\overline{z}^n V \subseteq \overline{z}^n \mathcal{O} \subseteq \overline{z}^{n-r} V$ for all $n$.

In case (i), choose some integer $m$ such that $p^m \geq r+1$: this implies that $\deg_V(\widehat{d}^{p^m}) \geq p^m \geq r+1$, and so $\widehat{d}^{p^m}(\overline{z}^n \mathcal{O}) \subseteq \widehat{d}^{p^m}(\overline{z}^{n-r} V) \subseteq \overline{z}^{n+1} V \subseteq \overline{z}^{n+1} \mathcal{O}$: or, in other words, $\deg_\mathcal{O}(\widehat{d}^{p^m}) \geq 1$.

In case (ii), choose some integer $m \geq r+1$. Then since $p\in\overline{z}V$, we have that $v_{\overline{z},V}(p)\geq 1$, so applying Lemma \ref{lem: raising to pth powers in the char 0 case} gives that $\deg_V(\widehat{\sigma}^{p^m} - \id) \geq m$. We conclude similarly that $\deg_\mathcal{O}(\widehat{\sigma}^{p^m} - \id) \geq 1$.

\textbf{Step 2.} \cite[\S 3.14, proof of Theorem C]{ardakovInv} tells us that $\overline{z}^{t}\in J(\mathcal{O})^2$ for a large enough integer $t$: fix such a $t$.

In case (i), let $\ell$ be such that $p^{\ell-m} \geq t$, so that $$\widehat{d}^{p^\ell}(\mathcal{O}) \subseteq (\widehat{d}^{p^m})^{p^{\ell-m}}(\O) \subseteq \overline{z}^{p^{\ell-m}} \mathcal{O} \subseteq \overline{z}^t \mathcal{O} \subseteq J(\mathcal{O})^2.$$ Now Lemma \ref{lem: checking compatibility of a derivation on O and J(O)} shows that $\deg(\widehat{d}^{p^\ell}) \geq 1$.

In case (ii), we can apply Lemma \ref{lem: raising to pth powers in the char 0 case} to show that $\deg_\O(\widehat{\sigma}^{p^\ell} - \id) = \deg_\O ((\widehat{\sigma}^{p^m})^{p^{\ell-m}} - \id) \geq \ell-m$ for any $\ell$: choosing $\ell\geq m+t$ allows us to conclude similarly that $(\widehat{\sigma}^{p^\ell}-\id)(\mathcal{O})$ is contained in $J(\mathcal{O})^2$, and hence $\deg_{J(\O)}(\widehat{\sigma}^{p^\ell}-\id) \geq 1$.
\end{proof}

\subsection{Invariant maximal ideals}

The results of the previous subsection are almost sufficient to complete the proof of Theorem \ref{B}. The only obstacle is the need to find a maximal ideal $M$ of $Q'$ that is $(\sigma,\delta)$-invariant. In general, a maximal ideal $M_1$ of $Q'$ will not be a $\sigma$-ideal, but will instead have a nontrivial $\sigma$-orbit $\{M_1, \dots, M_s\}$, so we cannot pass to the quotient $Q'/M_1$ without losing the skew derivation $(\sigma,\delta)$. 

Write $N := \bigcap_{i=1}^s M_i$, a minimal $\sigma$-prime ideal, and set $\widehat{Q} = Q'/N$, now a \emph{semisimple} artinian ring, and let $\widehat{\sigma}$ be the automorphism of $\widehat{Q}$ induced by $\sigma$. In this section, we will prove that after raising $\sigma$ to a sufficiently high $p$'th power, this construction produces a simple ring.

Now, even though $\widehat{Q}$ is $\widehat{\sigma}$-prime as defined, it may not be $\widehat{\sigma}^{p^m}$-prime for $m\geq 1$: that is, the $\sigma$-orbit $\{M_1, \dots, M_s\}$ may break up into more than one $\sigma^{p^m}$-orbit. This happens precisely when $p$ divides $s$, so we begin by showing that we can restrict our attention away from this case.

\begin{lem}\label{lem: size of orbit of maximal ideals can be made coprime to p}
There exists some $k\geq 0$ such that $M'_1 := M_1$ has $\sigma^{p^k}$-orbit $\{M'_1, M'_2, \dots, M'_t\}$ for some $t$ coprime to $p$.
\end{lem}

\begin{proof}
Suppose $s = ps'$. Write $\mathfrak{O} = \{1, 2, \dots, s\}$ and $\mathfrak{o}_i = \{pn + i: 0\leq n\leq s'-1\}$ for $1\leq i\leq p$: then $\mathfrak{O}$ is the disjoint union of the $\mathfrak{o}_1, \dots, \mathfrak{o}_p$. These are index sets for $\sigma$-orbits and $\sigma^p$-orbits respectively: that is, each set $\{M_j : j\in \mathfrak{o}_i\}$ (for $1\leq i\leq p$) is a $\sigma^p$-orbit, and their disjoint union is the $\sigma$-orbit $\{M_j : j\in \mathfrak{O}\}$.

So we will replace $\sigma$ by $\sigma^p$, and replace the $\sigma$-orbit $\{M_j : j\in \mathfrak{O}\}$ with the $\sigma^p$-orbit $\{M_j : j\in \mathfrak{o}_1\}$. This orbit has size $s'$, and so the result follows from induction.\end{proof}

Replacing $\sigma$ by $\sigma^{p^k}$, if necessary, we may assume henceforth that $p$ does not divide $s$.

Now, using the general theory of semisimple artinian rings, we see that $\widehat{Q}\cong A_1\times\cdots\times A_s$, where $A_i=Q'/M_i$ for each $i$. Furthermore, we may assume without loss of generality that $\sigma(M_i)=M_{i+1}$ for each $i$ (indices taken modulo $s$), and hence the action of $\widehat{\sigma}$ on $A_1\times\cdots\times A_s$ is given by $\widehat{\sigma}(a_1+M_1,\cdots,a_s+M_s)=(\sigma(a_s)+M_1,\sigma(a_1)+M_2,\cdots,\sigma(a_{s-1})+M_s)$.

Now, let us assume $\sigma-\mathrm{id}$ has positive degree with respect to $v_{z,U}$, i.e. $(\sigma,\sigma-\mathrm{id})$ is compatible with $v_{z,U}$ as before. Set $V = (U+N)/N \subseteq \widehat{Q}$, and write $\overline{z} = z + N \in \widehat{Q}$. We set up the remaining necessary objects and filtrations in advance:

\begin{itemize}[noitemsep]
\item $V_i\subseteq A_i$ be the image of $V$ under the $i$'th projection $\widehat{Q}\to A_i$,
\item $W:=V_1\times V_2\times\cdots\times V_s\subseteq\widehat{Q}$,
\item $\overline{z}_i \in V_i$ the image of $\overline{z}$ under the projection $\widehat{Q}\to A_i$ for each $i$,
\item $v_{\overline{z},V}$ the $\overline{z}$-adic filtration with respect to $V$ on $\widehat{Q}$ with level sets $\{\overline{z}^n V\}_{n\in\mathbb{Z}}$.
\end{itemize}

It follows from the proof of Proposition \ref{propn: from Q' to Q-hat, case Q-hat simple} that $\widehat{\sigma}-\id$ has positive degree with respect to $v_{\overline{z},V}$. Since each $A_i$ is simple, we are now back in the situation of the previous subsection, and we may choose a maximal order $\mathcal{O}_i\subseteq A_i$ equivalent to $V_i$ satisfying Properties \ref{props: maximal order O} (where $\widehat{Q}$ is replaced by $A_i$ in the statement). In particular, $\mathcal{O}_i\subseteq\overline{z}_i^{-r_i} V_i$ for some $r_i\geq 0$ for each $i$. 

So, write $\O = \O_1 \times \dots \times \O_s \subseteq A_1\times\cdots\times A_s\cong\widehat{Q}$, and set $r$ to be the maximum of the $r_i$ for $1\leq i\leq s$, so that $\mathcal{O}\subseteq\overline{z}_1^{-r}V_1\times\cdots\times\overline{z}_s^{-r} V_s$. Note that since $A_i=Q'/M_i$, it follows that $\overline{z}_i$ is the image of $z\in U$ modulo the maximal ideal $M_1$, and since $\widehat{Q}\cong A_1\times\cdots\times A_s$, it is clear that the image of $\overline{z}$ under this isomorphism is $(\overline{z}_1,\cdots,\overline{z}_s)$, and hence $\mathcal{O}\subseteq\overline{z}^{-r}W$.

Also, \cite[Proposition 3.2.4(ii)]{MR} implies that $V\subseteq W$ and the two are equivalent as orders in $\widehat{Q}$. This implies that $W\subseteq\overline{z}^{-a} V$ for some $a\geq 0$, so $\mathcal{O}\subseteq\overline{z}^{-(r+a)}V$, so we set $x:=r+a$ to get that $\mathcal{O}\subseteq\overline{z}^{-x}V$.

\begin{lem}
$\O$ is a maximal order equivalent to $V$ inside $\widehat{Q}$.
\end{lem}

\begin{proof}
Since $V\subseteq W\subseteq\mathcal{O}$ and each $\mathcal{O}_i$ is prime by Property \ref{props: maximal order O}(iii) applied to $V_i$, $\mathcal{O}$ is a maximal order by \cite[Proposition 5.1.5]{MR}, equivalent to $W$.
\end{proof}

Note that $J(\mathcal{O}) = J(\mathcal{O}_1) \times \dots \times J(\mathcal{O}_s)$, so we again consider the $J(\O)$-adic filtration $u$ on $\widehat{Q}$ with level sets $\{J(\O)^n\}_{n\in\mathbb{Z}}$. If we set $u_i$ as the $J(\mathcal{O}_i)$-adic filtration on $A_i$, then it follows that $$u(a_1+M_1,\cdots,u_s+A_s)=\min\{u_i(a_i+M_i):i=1,\cdots,s\},$$ i.e. $u$ is the product filtration.

It follows from \cite[\S 3.14, proof of Theorem C]{ardakovInv} that for each $i$ we may choose some large enough integer $t_i$ such that $\overline{z}_i^{t_i}\in J(\O_i)^2$, and setting $t = \max_i\{t_i\}$ we see that $$\overline{z}^t=(\overline{z}_1^t,\cdots,\overline{z}_s^t) \in J(\mathcal{O}_1)^2\times\cdots\times J(\mathcal{O}_s)^2 = J(\mathcal{O})^2.$$

Therefore, since $\widehat{\sigma}-\id$ has positive degree with respect to $v_{\overline{z},V}$, it follows from Lemma \ref{lem: raising to pth powers in the char 0 case} that $\widehat{\sigma}^{p^{t+x}}-\id$ has degree at least $t+x$ with respect to $v_{\overline{z},V}$, i.e. $$(\widehat{\sigma}^{p^{t+x}}-\id)(\overline{z}^mV)\subseteq\overline{z}^{m+t+x} V$$ for all $m\in\mathbb{Z}$.

In particular, $(\widehat{\sigma}^{p^{t+x}}-\id)(\mathcal{O})\subseteq (\widehat{\sigma}^{p^{t+x}}-\id)(\overline{z}^{-x}V)\subseteq\overline{z}^tV\subseteq J(\mathcal{O})^2$, and hence $\widehat{\sigma}^{p^{t+x}}-\id$ has positive degree with respect to the $J(\mathcal{O})$-adic filtration $u$ by Lemma \ref{lem: checking compatibility of a derivation on O and J(O)}, i.e. $u((\widehat{\sigma}^{p^{t+x}}-\id)(q))>u(q)$ for all $q\in\widehat{Q}$.

In particular, if we set $e_i=(0,\cdots,0,1,0,\cdots,0)\in A_1\times\cdots\times A_s\cong\widehat{Q}$, where the 1 is in the $i$'th position, then $u((\widehat{\sigma}^{p^{t+x}}-\id)(e_1))>u(e_1)=0$. But $\widehat{\sigma}^{n}(e_1)=e_{n+1}$, subscripts taken modulo $s$, so  $\widehat{\sigma}^{p^{t+x}}(e_1)=e_{p^{t+x}+1}$, and hence $u(e_{p^{t+x}+1}-e_1)>0$, and since $u$ is just the product filtration, this implies that $e_1$ and $e_{p^{t+x}+1}$ must share a common non-zero component, which is only possible if $p^{t+x}\equiv 0$ (mod $s$).

But since we are assuming that $p$ does not divide $s$, this is only possible if $s=1$, and hence $\widehat{Q}=A_1$ is a simple ring, and $N=M_1$ is maximal. So altogether, we have proved:

\begin{thm}\label{invariant maximal ideal}
If $\sigma$ is an automorphism of $Q'$ such that $(\sigma,\sigma-\mathrm{id})$ is compatible with $v_{z,U}$, then there exists $k\in\mathbb{N}$ and a maximal ideal $M$ of $Q'$ such that $M$ is $\sigma^{p^k}$-invariant.\qed
\end{thm}

\subsection{Proof of main results}

As in previous subsections, and in the statement of Theorem \ref{B}, $R$ will denote a \emph{prime} algebra over $\mathbb{Z}_p$, equipped with a complete, separated Zariskian filtration $w: R\to \mathbb{Z}\cup\{\infty\}$ such that $\gr_w(R)$ is commutative (and Noetherian), $w(p^n)=nw(p)\geq n$ for all $n\in\mathbb{N}$, and $(\gr_w(R))_{\neq 0}$ is not nilpotent. We also assume that $(\sigma,\delta)$ is compatible with $w$ and $\sigma\delta=\delta\sigma$.

\emph{Proof of Theorem \ref{B}.} Following \cite{ardakovInv}, our construction goes as follows.

\textbf{Step 1.} Let $Q'$ be the completion of $Q(R)$ with respect to the filtration $w'$ constructed in \S \ref{subsec: constructing w'}, an artinian ring by Property \ref{props: the filtration w'}(v). As we remarked below Properties \ref{props: the element z}, the induced skew derivation $(\sigma,\delta)$ on $Q'$ remains compatible with $w'$.

\textbf{Step 2.} Let $U$ be the positive part of $Q'$, i.e. $U = w^{-1}([0,\infty])$, and choose the regular normal element $z\in J(U)$ satisfying Properties \ref{props: the element z}. Then the $z$-adic filtration $v_{z,U}$ with respect to $U$ on $Q'$ satisfies Proposition \ref{propn: the filtration v_{z,U}}: that is, if we are in case (a) (where $R$ is an $\mathbb{F}_p$-algebra), then we have some $m$ such that $(\sigma^{p^m}, \delta^{p^m})$ is compatible with $v_{z,U}$, and if we are in case (b) (where $\delta=\sigma-\mathrm{id}$), then we have some $m$ such that $(\sigma^{p^m}, \sigma^{p^m} - \mathrm{id})$ is compatible with $v_{z,U}$.

\textbf{Step 3.} In both cases, we now know that $(\sigma^{p^m}, \sigma^{p^m} - \mathrm{id})$ is compatible with $v_{z,U}$, and so we may apply Theorem \ref{invariant maximal ideal} to choose a maximal ideal $M$ of $Q'$ such that $M$ is $\sigma^{p^k}$-invariant for some $k\geq m$. It follows trivially that $M$ is $(\sigma^{p^t},\sigma^{p^t}-\mathrm{id})$-invariant for $t = k$, which is what we need for case (b).

So assume we are in case (a). As $M$ is a minimal $\sigma^{p^k}$-invariant prime ideal, it is in particular a minimal $\sigma^{p^k}$-prime ideal, and so it follows from Theorem \ref{C} that there exists an ideal $J\supseteq M$ such that $J$ is $\sigma^{p^t}$-prime for some $t\geq k$, and $\delta^{p^{t}}(J)\subseteq J$. But $M$ is a maximal ideal of $Q'$, and hence $J=M$, meaning that $M$ is $(\sigma^{p^t},\delta^{p^t})$-invariant.

\textbf{Step 4.} Let $\widehat{Q}:=Q'/M$, a simple artinian ring. As $M$ is invariant under $(\sigma^{p^t}, \delta^{p^t})$ in case (a), there is an induced skew derivation $\left(\widehat{\sigma^{p^t}},\widehat{\delta^{p^t}}\right)$ on $\widehat{Q}$ in this case; similarly, there is an induced skew derivation $\left(\widehat{\sigma^{p^t}},\widehat{\sigma^{p^t}}-\id\right)$ on $\widehat{Q}$ in case (b).

Set $V = (U+M)/M$ and $\overline{z} = z+M$ as in \S \ref{subsec: construction of O}. Writing $v_{\overline{z}, V}$ for the quotient filtration of $v_{z,U}$ as in that subsection, we see by Proposition \ref{propn: from Q' to Q-hat, case Q-hat simple} that our induced skew derivation on $\widehat{Q}$ is compatible with $v_{\overline{z}, V}$ in both cases.

\textbf{Step 5.} Let $\mathcal{O}$ be a maximal order in $\widehat{Q}$ equivalent to $V$ as in Properties \ref{props: maximal order O}, and let $u$ be the $J(\mathcal{O})$-adic filtration on $V$. It now follows from \cite[Theorem C]{ardakovInv} that $\gr_u \widehat{Q}$ is a skew polynomial ring in one variable over a central simple algebra, and hence is prime and Noetherian. Since $Q(R)$ is simple, the composition $Q(R)\to Q'\to\widehat{Q}$ is injective, so we will denote the restriction of $u$ to $Q(R)$ again by $u$: we have seen that $\widehat{Q}$ is in fact the completion of $Q(R)$ with respect to this restriction, so $\gr_u Q(R)\cong \gr_u \widehat{Q}$ is prime and Noetherian as required. Also, the map $(R,w)\to (Q(R),u)$ is continuous by \cite[Theorem C]{ardakovInv}.

Furthermore, applying Proposition \ref{propn: raising to pth powers, case Q-hat simple}, we may find some $\ell$ such that the skew derivations $\left((\widehat{\sigma^{p^t}})^{p^\ell},(\widehat{\delta^{p^t}})^{p^\ell}\right)$ in case (a), resp. $\left((\widehat{\sigma^{p^t}})^{p^\ell},(\widehat{\sigma^{p^t}})^{p^\ell} - \mathrm{id}\right)$ in case (b), are compatible with $u$. But the restrictions of $(\widehat{\sigma^{p^t}})^{p^\ell}$ and $(\widehat{\delta^{p^t}})^{p^\ell}$ to $Q(R)$ are simply $\sigma^{p^N}$ and $\delta^{p^N}$ respectively, where $N = t + \ell$. The desired conclusions now follow.\qed

Finally, we conclude with a proof of our main result, which now follows easily from Theorem \ref{B} and the results from \S 2.

\emph{Proof of Theorem \ref{A}.} By Theorem \ref{B}, there exists a filtration $u$ on $Q(R)$ such that $\gr_u Q(R)$ is prime and Noetherian, and the map $(R,w)\to (Q(R),u)$ is continuous. 

If $R$ is an $\mathbb{F}_p$-algebra then there exists $N\in\mathbb{N}$ such that $(\sigma^{p^N},\delta^{p^N})$ is compatible with $u$, and if $\delta=\sigma-\mathrm{id}$ then there exists $N\in\mathbb{N}$ such that $(\sigma^{p^N},\sigma^{p^N}-\id)$ is compatible with $u$.

Therefore, it follows from Theorem \ref{thm: key for Thm A} that if $R$ is an $\mathbb{F}_p$-algebra then the skew power series ring $R^b[[x^{p^N};\sigma^{p^N},\delta^{p^N}]]$ is prime, and if $\delta=\sigma-\mathrm{id}$ then $R^b[[(x+1)^{p^N} - 1;\sigma^{p^N},\sigma^{p^N}-\id]]$ is prime, and this completes the proof.\qed 

\subsection{Applications}\label{subsec: applications}

\begin{ex}\label{ex: iterated local SPS ring}
Suppose that $A$ is an iterated local skew power series ring (in the sense of \cite{woods-SPS-dim}) with maximal ideal $\mathfrak{m}_A$, i.e. $A = k[[x_1]][[x_2; \sigma_2, \delta_2]]\dots [[x_n; \sigma_n, \delta_n]]$, over $k = \mathbb{F}_p$ or $\mathbb{Z}_p$, such that $\gr_f(A)$ is an iterated \emph{commutative} polynomial ring over $\mathbb{F}_p$. (Here, $f$ can be taken to be the $\mathfrak{m}_A$-adic filtration or other related filtrations -- see the remark below.) Suppose $(\tau, \varepsilon)$ is a skew derivation on $A$ which is compatible with the filtration. Write $S = A[[y; \tau, \varepsilon]]$. Then, given any prime ideal $P\lhd A$ which is $(\tau,\varepsilon)$-invariant, we will show that there exists a finite-index subring $S_N\subseteq S$ (containing $A$) such that $PS_N$ is a prime ideal of $S_N$.

\textit{Remark.} The assumption that $\gr_f(A)$ is commutative is reasonably mild: for example, the graded ring associated to the natural $\mathfrak{m}_A$-adic filtration will be commutative if all of the $(\sigma_i, \delta_i)$ are compatible in the sense of Definition \ref{defn: compatible}. Alternatively, we can construct such a filtration $f$ if $A$ is \emph{triangular} -- see \cite[\S 0.3]{woods-SPS-dim}.

We proceed as follows. It is known already that $\mathfrak{m}_AS$ is a prime ideal of $S$ \cite{woods-SPS-dim}, so take any non-maximal prime ideal $P\lhd A$ which is preserved by $(\tau, \varepsilon)$. We will show that $R = A/P$ satisfies the hypotheses of Theorem \ref{A} by induction on $n$ -- in particular, we need only show that $\gr(R)_{\neq 0}$ contains a non-nilpotent element.

So we begin with a base case. If $A = \mathbb{F}_p[[x_1]]$ or $\mathbb{Z}_p$, then (as $P$ is assumed non-maximal) we must have $P = 0$. Hence $\gr(R) \cong \mathbb{F}_p[X]$, and we are done by taking the element $X$.

For larger $n$, write $B = k[[x_1]][[x_2;\sigma_2, \delta_2]]\dots [[x_{n-1}; \sigma_{n-1},\delta_{n-1}]]$ (again a local ring with maximal ideal $\mathfrak{m}_B$), so that $A = B[[x_n; \sigma_n, \delta_n]]$. In this context, we know that $(P\cap B)A$ is a two-sided ideal of $A$, and we write $\overline{a}$ for the image of $a$ under the map $A\to \overline{A} := A/(P\cap B)A$. So in the case where $P\cap B = \mathfrak{m}_B$, we know that $R = A/P$ is a prime quotient ring of $\overline{A} \cong \mathbb{F}_p[[\overline{x}_n]]$ and that $P$ is non-maximal, so $R \cong \mathbb{F}_p[[\overline{x}_n]]$ and $\gr(R) \cong \mathbb{F}_p[X]$ as above.

Hence we assume that $P\cap B \neq \mathfrak{m}_B$. Continue to denote the filtration on $A$ by $f$, and assume it takes values inside $e\mathbb{N}\cup\{\infty\}$ for some $e > 0$. Note that $f$ induces separated filtrations $f_P$ on $A/P$ (and on the subgroup $(B+P)/P$) and $\overline{f}$ on $B$ defined by
\begin{align*}
f_P(a+P) &= \sup\{f(a+r): r\in P\},\\
\overline{f}(b+P\cap B) &= \sup\{f(b+r): r\in P\cap B\}
\end{align*}
for all $a\in A, b\in B$. It is clear that, for any $b\in B$, we have $\overline{f}(b+P\cap B) \leq f_P(b+P)$.

We may now invoke \cite[Theorem 3.17(ii)]{letzter-noeth-skew}, which holds with an identical proof even under our slightly weaker hypotheses. This tells us that $P\cap B$ must be a $(\sigma_n, \delta_n)$-prime ideal of $B$. In particular, it also cannot contain any \emph{power} of $\mathfrak{m}_B$, as $\mathfrak{m}_B$ is a $(\sigma_n, \delta_n)$-invariant ideal. In particular, given any large $N\in\mathbb{N}$, we can find an element $b\in \mathfrak{m}_B^N \setminus (P\cap B)$, which must satisfy $b+ P \neq 0$ and $f_P(b+P) \geq \overline{f}(b+P\cap B) \geq eN$. In other words: $(B+P)/P \leq A/P = R$ contains elements of arbitrarily large value under $f_P$, which ensures that $\gr(R)$ cannot be finite-dimensional as an $\mathbb{F}_p$-vector space. This implies that $\gr(R)_{\neq 0}$ contains a non-nilpotent element, and so we may apply Theorem \ref{A}.
\end{ex}

\begin{ex}\label{ex: iwasawa algebras}
Let $G$ be a uniform pro-$p$ group in the sense of \cite{DDMS}, $H$ a closed isolated normal subgroup satisfying $G/H \cong \mathbb{Z}_p$, and $k = \mathbb{F}_p$ or $\mathbb{Z}_p$. Then we may consider the completed group algebra (\emph{Iwasawa algebra}) $kH$. It is known (see e.g. \cite{ardakovbrown, venjakob}) that $kG$ may be written as $kH[[x; \sigma, \delta]]$, where $x = g-1$, $\sigma$ is conjugation by $g$ and $\delta = \sigma - \mathrm{id}$ for any $g\in G\setminus H$ such that $G = \overline{\langle H, g\rangle}$.

Now, for every $G$-invariant (= $(\sigma, \delta)$-invariant) prime ideal $P\lhd kH$, we may consider the ring $kG/PkG$ as a skew power series ring over $kH/P$ by \cite[Lemma 3.14(iv)]{letzter-noeth-skew}. If $H$ is solvable, then $kH/P$ satisfies the hypotheses of Example \ref{ex: iterated local SPS ring} (and hence of Theorem \ref{A}) whenever $P$ is not maximal; if not, a similar argument by considering the $H$-prime ideal $P\cap kK$, where $K$ is a maximal proper closed isolated normal subgroup of $H$, will suffice.

\textit{Remark.} In this context, Theorem \ref{A} tells us that the extension of $P$ to the ring $kH[[x_{(N)}; \sigma_{(N)}, \delta_{(N)}]]$ is prime. But note that $x_{(N)} = g^{p^N} - 1$, $\sigma_{(N)}$ is conjugation by $g^{p^N}$, and $\delta_{(N)} = \sigma_{(N)} - \mathrm{id}$: that is, if we set $G_P$ to be the (open, normal) subgroup of $G$ defined by $G_P := \overline{\langle H, g^{p^N}\rangle}$, this ring is just the completed group algebra $kG_P$, i.e. the extension $P\cdot kG_P$ is prime. Moreover, $kG = kG_P * (G/G_P)$, and $G/G_P \cong \mathbb{Z}/p^N\mathbb{Z}$: hence the ideal $P\cdot kG$ may be viewed as $(P\cdot kG_P)*(\mathbb{Z}/p^N\mathbb{Z})$ as in the Introduction, and we may conclude Corollary \ref{cor: iwasawa algebra is semiprime} depending on the characteristic of $kG/P$.
\end{ex}

\bibliography{biblio}
\bibliographystyle{plain}

\end{document}